\documentclass[11pt,reqno,a4paper]{amsart}
\usepackage[margin=1.46cm]{geometry}
\usepackage[usenames]{color}
\usepackage{amsmath,pdfsync,verbatim,graphicx,epstopdf,enumerate}
\usepackage[normalem]{ulem}
\usepackage[colorlinks=true]{hyperref}
\usepackage{cancel}
\usepackage[framemethod=tikz]{mdframed}
\hypersetup{allcolors=blue}

\numberwithin{equation}{section}
\newcommand{\I}{\mathrm{i}}

\newcommand{\wh}{\widehat}

\newcommand{\lb}{\left(}

\newcommand{\rb}{\right)}

\newcommand{\Beq}{\begin{equation}}
	\newcommand{\Eeq}{\end{equation}}
\newcommand{\beq}{\begin{equation*}}
	\newcommand{\eeq}{\end{equation*}}
\newcommand{\bal}{\begin{align}}
	\newcommand{\eal}{\end{align}}

\usepackage{mathtools}

\allowdisplaybreaks

\usepackage[notref,notcite]{}

\newcommand{\B}{\beta}
\newcommand{\bp}{\begin{prob}}
	\newcommand{\ep}{\end{prob}}
\newcommand{\bpr}{\begin{proof}}
	\newcommand{\epr}{\end{proof}}
\renewcommand{\o}{\omega}

\newcommand{\bel}[1]{\begin{equation}\label{#1}}
	\newcommand{\ee}{\end{equation}}

\newtheorem{theorem}{Theorem}[section]

\newtheorem{lemma}[theorem]{Lemma}

\theoremstyle{definition}

\newtheorem{remark}[theorem]{Remark}

\newcommand{\Rn}{\mathbb{R}^n}

\newcommand{\D}{\mathrm{d}}
\newcommand{\Lc}{\mathcal{L}}
\newcommand{\Rb}{\mathbb{R}}

\newcommand{\A}{\alpha}

\newcommand{\ve}{\varepsilon}

\newcommand{\Bb}{\mathbb{B}}
\usepackage[usenames]{color}
\usepackage{amsmath,pdfsync,verbatim,graphicx,epstopdf,enumerate}

\newcommand{\Rc}{\mathcal{R}}

\newcommand{\Nb}{\mathbb{N}}

\newcommand{\Sb}{\mathbb{S}}
\newcommand{\Sn}{\mathbb{S}^{n-1}}

\renewcommand{\o}{\omega}

\title[Range characterization for SMT in even dimensions]{A  simple range characterization for spherical mean transform in even dimensions}
\author[Agrawal,~ Ambartsoumian, ~Krishnan,~ Singhal]{Divyansh Agrawal$^\ast$, Gaik Ambartsoumian$^\dagger$, Venkateswaran P.\ Krishnan$^\ast$ and Nisha Singhal$^\ast$}

\address {$^{\ast}$ Centre for Applicable Mathematics, Tata Institute of Fundamental Research, Bangalore, India
\newline
E-mail:{\tt\  agrawald@tifrbng.res.in, vkrishnan@tifrbng.res.in, nisha2020@tifrbng.res.in}
\newline
Orcid:{\tt\ 0009-0003-5125-0640, 0000-0002-3430-0920, 0009-0006-3005-1986}}
\address{$^\dagger$ Department of Mathematics, The University of Texas at Arlington, Texas, USA
\newline
E-mail:{\tt \ gambarts@uta.edu}
\newline
Orcid:{\tt\ 0000-0002-1462-9964}}
\date{\today}

\begin{document}

\begin{abstract}
The paper presents a new and simple range characterization for the spherical mean transform of functions supported in the unit ball in even dimensions. It complements the previous work of the same authors, where they solved an analogous problem in odd dimensions. The range description in even dimensions consists of symmetry relations, using a special kind of elliptic integrals involving the coefficients of the spherical harmonics expansion of the function in the range of the transform.  The article also introduces a pair of original identities involving normalized Bessel functions of the first and the second kind. The first result is an integral cross-product identity for Bessel functions of integer order, complementing a similar relation for Bessel functions of half-integer order obtained in the aforementioned work of the same authors. The second result is a new Nicholson-type identity. Both of these relations can be considered as important standalone results in the theory of special functions. Finally, as part of the proof of one of the theorems, the authors  derive an interesting equality involving elliptic integrals, which may be of independent interest.

\end{abstract} 

\maketitle


\section{Introduction}

The spherical mean transform (SMT) maps a function of $n$ variables to its normalized integrals over all spheres in a given $n$-dimensional family. Some typical examples of such families include the spheres with centers restricted to a hypersurface, the spheres tangent to a hypersurface, and the spheres passing through a given point. 

The study of SMTs has a long history, going back to the classical works of Courant and Hilbert \cite{CH_Book}, and John \cite{John-book} on partial differential equations. Later, such transforms were investigated in relation to various problems arising in approximation theory, integral geometry, and tomography (e.g. see \cite{agranovsky1996approximation,  agranovsky1996injectivity, agranovsky1999conical, ref:AmbKuch, And, aramyan2020recovering, denisjuk1999integral, Fawcett} and the references therein). Of particular interest, motivated by imaging applications, has been the SMT integrating a function over spheres with centers restricted to the unit sphere. 
Numerous important results have been obtained about its inversion (e.g. see \cite{ ambartsoumian2014exterior, Finch-Haltmeir-Rakesh_even-inversion, Finch-P-R, K, nguyen2009family, Norton-circle, norton1981ultrasonic, R}), range description (\cite{Agranovsky-Finch-Kuchment-range, Agranovsky-Kuchment-Quinto, agranovsky2010range, AAKN1, AAKN2, ref:AmbKuch-range, finch2006range, Kuchmment_2025}), microlocal analysis and incomplete data problems (e.g. see \cite{ Ambartsoumian2018, Ambartsoumian-Zarrad-Lewis,  Ambartsoumian2015,  xu2004reconstructions}).

In this paper we derive a new range characterization for the SMT in the spherical geometry of data acquisition in even dimensions.  Typically, the range of a generalized Radon transform, like SMT, satisfies infinitely many independent conditions in standard function spaces. A complete description of the range is of significant importance for multiple reasons. In imaging applications, the transform data often corresponds to the measurements of the unknown (image) function, which may be incomplete, contaminated by noise, or corrupted by measurements errors. The knowledge of the range conditions can be helpful in filling in the missing data, suppressing the noise, and discovering hardware imperfections. In theoretical studies of the transform, the range description is often necessary in constructive proofs, when one would like to produce a function in the range with prescribed properties, e.g. specific restrictions on its support. 

The first complete range characterizations for the SMT were described in a pair of articles published in 2006. The authors of \cite{ref:AmbKuch-range} solved the problem in 2D. In addition to the moment-type conditions that are typical for generalized Radon transforms, their range description also included a set of less standard ``vanishing'' conditions. Namely, the Hankel transform of each (modified) $n$-th Fourier coefficient of the data function should vanish at every zero of the Bessel function $J_n$ away from the origin. They also showed that in the case of the standard Radon transform, one can formulate somewhat similar vanishing conditions that are equivalent to its moment-type conditions. However, it was not clear at the time whether in the case of SMT the vanishing conditions also imply its moment-type conditions.

The authors of \cite{finch2006range} gave a complete characterization of the SMT in arbitrary odd dimensions, using its connection to the solution of the backward initial-boundary value problem for the wave equation. This approach takes advantage of the Huygen’s principle for the wave equation in odd dimensions, and the fact that in odd dimensions the solution of the initial value problem for the wave equation is given by a local operator applied to the SMT. The range description presented in this paper included vanishing conditions similar to those obtained in \cite{ref:AmbKuch-range} (with the Hankel transform replaced by the cosine transform), but did not require any moment type conditions.

A pair of subsequent publications \cite{Agranovsky-Finch-Kuchment-range, Agranovsky-Kuchment-Quinto} generalized the results of \cite{ref:AmbKuch-range, finch2006range}. The authors of \cite{Agranovsky-Kuchment-Quinto} gave three different types of range descriptions in arbitrary dimensions, combining the moments conditions with (a) vanishing conditions on the Fourier-Hankel transform of the spherical harmonics, (b) an orthogonality condition for the solution of a related interior problem of the Darboux equation, (c) an orthogonality condition on the data function with the eigenfunction of the Dirichlet Laplacian in the unit ball. Later in \cite{Agranovsky-Finch-Kuchment-range},  it was shown that the moment conditions of the range descriptions in \cite{Agranovsky-Kuchment-Quinto} are redundant, since they are implied by the other conditions, and can be dropped in all dimensions.

Despite the great theoretical value of all these range characterizations, their complexity made them of limited practical use both in imaging applications and in constructive proofs of various analytical properties of the SMT. Consequently, the search continued for simpler necessary and sufficient conditions for a function to be in the range of that transform. 

In the case of odd dimensions, such conditions were described in our recent work \cite{AAKN1}, in the form of symmetry relations of certain differential operators acting on the coefficients of the spherical harmonics expansion of the data function. This novel range characterization has multiple advantages over the previous ones. First of all, its simplicity allows one to construct functions in the range of the SMT with various prescribed features, e.g. a given support. As a direct application of that idea, we showed in \cite{AAKN1} that the so-called unique continuation property does not hold for the SMT in odd dimensions. In a follow-up work \cite{AAKN2}, we used our range description to disprove a conjecture of Rubin \cite{R} relating the SMT, its backprojection operator, and the Riesz potential, as well as address the associated problem of describing the null space of the backprojection operator. From the point of view of imaging applications, our symmetry relations allow one to (synthetically) double the measured data in the radial variable. This can potentially lead to improved quality of reconstructed images compared to the techniques using radially incomplete data (e.g. see \cite{roy2015efficient}). Another interesting
result of \cite{AAKN1} was the discovery of remarkable integral cross product identities for the spherical Bessel functions of the first and second kind, which can be of significant value in the theory of special functions. They also shed light and clarify the vanishing conditions involving the Hankel transform that appeared in the earlier range characterizations. 

In general, the range description of \cite{AAKN1} allows one to generate a function $g(p,t)\in\Sb^{n-1}\times (0,2)$  in the range of the SMT by a ``symmetric continuation'' of an arbitrary smooth function $g(p,t)$ defined on one half of the domain, i.e. $t\in(0,1]$ or $t\in[1,2)$, as long as the continuation through $t=1$ is smooth. The latter requirement is relatively straightforward to verify in specific cases. For example, this condition is satisfied if the function defined on one half of the domain vanishes when $t$ is close to 1. A different range characterization of the SMT (or equivalently, of the free space wave operator)  using half-data was recently given in \cite{Kuchmment_2025}. That description is implicit and requires expanding into a series the exterior Radon transform of the solution of an exterior initial/boundary value problem for a wave equation.

The main thrust of the current work is the derivation of an even-dimensional analog of the range description obtained in \cite{AAKN1} in the form of symmetry relations. As one could expect from the properties of the SMT in even dimensions, the symmetry relations here are of non-local character, and are expressed in the form of integral equations. As part of our analysis, we deduce a new set of integral cross product identities for Bessel functions of integer order, complementing similar identities for Bessel functions of half-integer order obtained in the case of odd dimensions. These relations can be considered as important standalone results in the theory of special functions. They also illuminate the vanishing conditions involving the Hankel transform that appear in the earlier range characterizations of the SMT in even dimensions. In addition to that, we derive a Nicholson-type identity, expressing the cross product of Bessel functions of first and second kind in terms of an integral involving only the Bessel function of the first kind. Finally, as part of one of our proofs, we obtain a curious equality involving elliptic integrals, which may be of independent interest.

\vspace{2mm}

The rest of the paper is organized as follows. Section \ref{sec:main} contains the statements of the main results. In Section \ref{sec:prelims} we present formal definitions and recall some known results used in the article. The proofs of the main results are given in Section \ref{sec:proofs}.


\section{Main results}\label{sec:main}
We now state the main results of this paper. As with our previous work in odd dimensions \cite{AAKN1}, we first derive a range description for the radial case and then use it to prove the statement for the general case. These range characterizations are formulated as Theorems \ref{range} and \ref{Thm1.4} below.
\begin{theorem}[Range characterization - radial case]\label{range}
    Let $\Bb$ denote the unit ball in $\mathbb{R}^n$ for even $n \geq 2$ and $\A=\frac{n-2}{2}$. A function $g \in C_c^\infty ((0,2))$ is representable as $g = \Rc f$ for a radial function $f \in C_c^\infty(\Bb)$ if and only if $h(t) \coloneqq t^{n-2} g(t)$ satisfies for each $0<s<1$,
    	\Beq\label{RC-Radial}
    	\int\limits_{0}^{1-s} \frac{th(t)}{t^{2\A}} \Big{\{}\left[(1+t)^2-s^2\right]\left[(1-t)^2 - s^2\right]\Big{\}}^{\A- \frac{1}{2}} \D t = \int\limits_{1+s}^{2} \frac{th(t)}{t^{2\A}} \Big{\{}\left[(1+t)^2-s^2\right]\left[(1-t)^2 - s^2\right]\Big{\}}^{\A- \frac{1}{2}} \D t.
    	\Eeq    
    \end{theorem}
\begin{theorem}[Range characterization - general case]\label{Thm1.4}
    Let $\Bb$ denote the unit ball in $\mathbb{R}^n$ for an even $n \geq 2$. A function $g \in C_c^\infty (\Sb^{n-1}\times (0,2))$ is representable as $g = \Rc f$ for $f\in C_c^\infty(\Bb)$ if and only if  for each $(m,l), m\geq 0, 0\leq l\leq d_m$, $h_{m,l}(t)=t^{n-2}g_{m,l}(t)$ satisfies the following two conditions:
    \begin{itemize}
        \item there is a function $\phi_{m,l}\in C_c^{\infty}((0,2))$ such that 
        \[
            h_{m,l}(t)= D^{m} \phi_{m,l}(t), \mbox{ where } D=\frac{1}{t}\frac{\D}{\D t},
        \]
        \item $\phi_{m,l}(t)$ satisfies the following: For $0< s< 1$,
	\[
    \begin{aligned} 
	\int\limits_{0}^{1-s} &\frac{t\phi_{m,l}(t)}{t^{2(m+\A)}} \Big{\{}\left[(1+t)^2-s^2\right]\left[(1-t)^2 - s^2\right]\Big{\}}^{m+\A- \frac{1}{2}} \D t \\
    &= \int\limits_{1+s}^{2} \frac{t\phi_{m,l}(t)}{t^{2(m+\A)}} \Big{\{}\left[(1+t)^2-s^2\right]\left[(1-t)^2 - s^2\right]\Big{\}}^{m+\A- \frac{1}{2}} \D t.
    \end{aligned} 
	\]   
    \end{itemize}
\end{theorem}


As part of the proof of the sufficiency in Theorems \ref{range} and \ref{Thm1.4}, we derive a couple of original and interesting identities involving normalized Bessel functions of the first and the second kind.
The first result is an integral cross-product identity for Bessel functions of integer order, complementing a similar relation for Bessel functions of half-integer order, obtained in the odd dimensional case in our earlier work \cite{AAKN1}.
\begin{theorem}\label{CP-Identity}
        Let $\alpha\in\mathbb{N}\cup \{0\}$. Assume that for each $0<s<1$ the function $h\in C_c^{\infty}((0,2))$ satisfies equation \eqref{RC-Radial}. 
        Then for any $\lambda>0$, the following integral cross product identity holds:
        \[
	y_{\A}(\lambda)\int\limits_0^{2}th(t) j_{\A}(\lambda t) \D t =j_{\A}(\lambda)\int\limits_0^{2}th(t) y_{\A}(\lambda t) \D t. 
	\]
    \end{theorem}
The second relation involving the normalized Bessel functions is a Nicholson-type identity, similar to the result of \cite{Hrycak-Schmutzhard}, which agrees with our result for $\A=0$.

\begin{theorem}\label{Nicholson-type formula} Let $\alpha\in\mathbb{N}\cup \{0\}$. For all complex numbers $z, w$, $z\neq w$ with $|\mbox{arg}(z)|<\pi$, $|\mbox{arg}(w)|<\pi$, consider the function
\Beq\label{Eqss3.1}
y(z,w)=\int\limits_{C} j_{\A}(\sqrt{z^2+w^2-2zw\cosh \zeta}) \sinh^{2\A}\zeta \D \zeta-\sum\limits_{j=0}^{\A-1} a_{j} D_{s}^{j} f(0),
\Eeq
where 
\[
f(s)=\frac{\A \,\mathrm{sgn}(|z|-|w|)}{2^{2\A-2}}\frac{\Bigg{\{}(z^2-w^2)^2-2s^2(z^2+w^2)+s^4\Bigg{\}}^{\frac{2\A-1}{2}}}{z^{2\A}w^{2\A}},
\]
and $a_{j}'s$ are defined iteratively by $a_0=1$ and $a_{j}=2(\A-j)a_{j-1}$. In \eqref{Eqss3.1}, $C$ is any contour joining $0$ and $\ln z-\ln w$. Then 
\[
y(z,w)=C(\A)\lb j_{\A}(z) y_{\A}(w)-y_{\A}(z) j_{\A}(w)\rb,
\]
where $C(\A)$ is a non-zero constant only depending on $\A$. The summation in \eqref{Eqss3.1} is understood as contributing zero sum for $\A=0$.
\end{theorem}


 The proof of necessity in Theorems \ref{range} and \ref{Thm1.4} uses an elegant equality involving elliptic integrals which, to the best of our knowledge, has not appeared in prior literature. We state this result below.

\begin{theorem}\label{Elliptic_Integrals} Let us denote
	\begin{align*}
	P(t) = (a-t)(t-b)(t-c)(t-d) = (a-t)(b-t)(c-t)(t-d),
	\intertext{where $a=(1+u)^2,\, b=(1+s)^2,\, c=(1-s)^2,\, d=(1-u)^2$ with  $0<s<u<1$.}
 \end{align*} 
        Then for any real $\B>-1$,
    	\Beq\label{Eqss3.4} 
    	\int\limits_b^a \frac{P^\B}{t^{\B}\sqrt{t}}\D t=  \int\limits_d^c \frac{P^\B}{t^{\B}\sqrt{t}}\D t.
    	\Eeq
    \end{theorem}


\section{Preliminaries}\label{sec:prelims}
In this section, we collect basic definitions and some technical results needed for our discussions. 
\subsection{Spherical mean transform} The transform considered here integrates a function supported in the unit ball over spheres, with centres restricted to the unit sphere. More precisely, for $f \in C_c^\infty (\Bb)$, the spherical mean transform is defined as 
\[
\Rc f (p,t) = \frac{1}{\o_n} \int\limits_{\Sn} f(p+t\theta) \D S(\theta), \quad (p,t) \in \Sn \times [0,2], 
\]
where $\D S(\theta)$ denotes the surface measure on the unit sphere and the normalizing constant $\o_n$ denotes its surface area. Note that we only need to  consider $0<t<2$ due to the support restriction on $f$.

\subsection{Spherical harmonics} These are restrictions of homogeneous harmonic polynomials to the unit sphere in $\Rn$. It is well known that for a fixed $m \geq 0$, there are 
\[
d_m = \frac{(2m+n-2) (n+m-3)!}{m! (n-2)!}, \quad d_0 = 1,
\]
linearly independent spherical harmonics of degree $m$, which we denote by $\{Y_{m,l}\}$, for $0\leq l\leq d_m$. These form a complete orthonormal basis. One of the most useful facts that we need about the spherical harmonics is the following lemma \cite{Claus_Muller}.
\begin{lemma}[Funk-Hecke] 
    For a smooth function $h$ on $[-1,1]$, we have
    \[
    \int\limits_{\Sb^{n-1}} h(\theta \cdot \o) Y_{m,l}(\o) \D S(\o) = c(n,m) Y_{m,l}(\theta),
    \]
    where
    \[
    c(n,m) = \frac{|\Sb^{n-2}|}{C_m^{\A}(1)} \int\limits_{-1}^1 h(t) C_m^{(n-2)/2}(t) (1-t^2)^{(n-3)/2} \D t.
    \]
    Here $C_m^{\A}$ denote the Gegenbauer polynomials of degree $m$ and order $\A > -1/2$. These are a family of orthogonal polynomials on $[-1,1]$, with the weight $(1-t^2)^{\A -1/2}$, given by the Rodrigues formula:
    \[
    C_m^\A (t) = \frac{(-1)^m}{2^mm!} \frac{\Gamma\lb \A+\frac{1}{2}\rb \Gamma \lb m+2\A \rb}{\Gamma(2\A)\Gamma\lb \A+m+\frac{1}{2}\rb } (1-t^2)^{-\A+\frac{1}{2}} \frac{\D^m}{\D t^m} \Big{\{}(1-t^2)^{m+\A-\frac{1}{2}}\Big{\}}.
    \]
\end{lemma}
Any function $f \in C_c^\infty(\Rn)$ has the spherical harmonics expansion
\[
f(x) = \sum\limits_{m=0}^\infty \sum\limits_{k=1}^{d_m} f_{m,l}(|x|) Y_{m,l}(x/|x|), \quad f_{m,l} (|x|) = \int\limits_{\Sn} f(|x|\theta) \overline{Y}_{m,l}(\theta) \D S(\theta).
\]

Viewing $(p,t) \in \Sn \times \Rb_+$ as polar coordinates of the point $tp$, the function $g(p,t) = \Rc f(p,t)$ can also be expanded in terms of spherical harmonics
\[
g(p,t) = \sum\limits_{m=0}^\infty \sum\limits_{l=1}^{d_m} g_{m,l} (t) Y_{m,l}(p), \quad g_{m,l}(t) = \int\limits_{\Sn} g(p,t) \overline{Y}_{m,l}(p) \D S(p).
\]

The advantage of considering the spherical harmonics expansions of $f$ and $\Rc f$ is that the integral equation reduces to an infinite system of decoupled one-dimensional equations. More precisely, the $(m,l)$-th coefficient of $\Rc f$ depends only on the $(m,l)$-th coefficient of $f$, given by the following expression (see \cite{Salman_Article}):
\[
(\Rc f)_{m,l} (t) = \frac{\o_{n-1}}{\o_n t^{n-2}C_m^{\A}(1)} \int\limits_{|1-t|}^1 u^{n-2} f_{m,l}(u) C_m^{\frac{n-2}{2}} \left( \frac{1+u^2-t^2}{2u} \right) \left\{ 1- \frac{(1+u^2-t^2)^2}{4u^2} \right\}^{\frac{n-3}{2}} \D u.
\]

\subsection{Fa\`a di Bruno formula for \boldmath{$D$} derivatives} Recall that $D= \frac{1}{t} \frac{\D}{\D t}$. We present the expression for higher $D$-derivatives of composition of two functions of one real variable in terms of the derivatives of individual functions; see \cite{Krantz-Parks_primer}.
\begin{lemma}[Fa\'a di Bruno]\label{L2.2}
    Let $F,G \in C^\infty(\Rb)$. Then
    \[
    D^k \left( F (G(t)) \right) = \sum\limits_{j=1}^k F^{(j)} (G(t)) B_{k,j} (DG(t), \dots, D^{k-j+1}G(t)),
    \]
    where $F^{(j)}$ denote the usual derivatives of $F$, and $B_{k,j}$ denote the Bell's polynomials defined as
    \[
    B_{k,j}(x_1, \dots, x_{k-j+1}) = \sum \frac{k!}{i_1! \dots i_{k-j+1}!} \lb \frac{x_1}{1!}\rb^{i_1} \dots \lb \frac{x_{k-j+1}}{(k-j+1)!}\rb^{i_{k-j+1}},
    \]
    the sum being taken over all sequences $i_1, \dots, i_{k-j+1}$ of non-negative integers such that
    \begin{align*}
        i_1 + \dots + i_{k-j+1} &= j,\\
        i_1 + 2 i_2 + \dots + (k-j+1) i_{k-j+1} &= k.
    \end{align*}
\end{lemma}

We need the following special case of the above formula when the inner function $G$ satisfies $D^jG = 0$ for $j \geq 3$. 
\begin{lemma}\cite{AAKN1}\label{L2.3}
    Let $F,G \in C^\infty (\Rb)$ such that $D^j G = 0$ for $j \geq 3$. Then
    \[
    D^k F(G(t)) = \sum\limits_{j \geq k/2}^k \frac{k!}{(2j-k)! (k-j)! 2^{k-j}} F^{(j)}(G(t)) (DG(t))^{2j-k} (D^2G(t))^{k-j}.
    \]
\end{lemma}
\subsection{Bessel functions} 
A central object in our study is the Bessel function. In this section, we fix our notation for the Bessel functions of the first and second kinds. 

For $\alpha \in \Rb \setminus \{-1,-2,\dots\}$, define the  Bessel function of first kind of order $\A$ to be 
\[
J_\A (x) = \lb \frac{x}{2}\rb^\A \sum\limits_{n=0}^\infty \frac{(-1)^n \lb \frac{x}{2}\rb^{2n}}{n! \Gamma(n+\A+1)}, \quad x \in (0,\infty),
\]
extended to negative integer orders by
\[
J_{-p}(x) = (-1)^p J_p (x).
\]
These are solutions to the following second order Bessel differential equation, 
\[
y'' + \frac{1}{x} y' + (1-\frac{\A^2}{x^2}) y = 0, \quad x \in (0,\infty).
\]
The normalized Bessel function of first kind of order $\A > -1/2$ is defined by
\[
j_\A(x) = 2^\A \Gamma(\A+1) \frac{J_\A (x)}{x^\A},
\]
and these are solutions of the normalized Bessel equation:
\[
y'' + \frac{2\A+1}{x} y' +y=0.
\]
They also satisfy the following derivation formula:
\begin{equation}\label{der}
    \lb \frac{1}{x} \frac{\D}{\D x} \rb^k (j_\A(x)) = \frac{(-1)^k \Gamma(\A+1)}{2^k \Gamma(k+\A+1)} j_{\A+k}(x), \quad \text{for all}~k \in \Nb.
\end{equation}
We will need the Fourier-Bessel transform of $g_{m,l}$ in the analysis below. This is defined as 
\Beq\label{FBT}
\wh{g}_{m,l}(\lambda)=\int\limits_0^{\infty}g_{m,l}(t)j_{\frac{n}{2}-1}(\lambda t) t^{n-1} \D t.
\Eeq

Using the Bessel functions of the first kind, we define the Bessel functions of the second kind by
\begin{equation}
    Y_\A (x) = \begin{cases}
        \frac{1}{\sin{\A \pi}} [\cos{\A \pi} J_\A(x) - J_{-\A}(x)], \quad &\A \notin \mathbb{Z},\\
        \lim\limits_{\A \to n} \frac{1}{\pi} \left[\frac{\partial}{\partial \A} J_\A (x) - (-1)^n \frac{\partial}{\partial \A} J_{-\A}(x)  \right], \quad &\A = n \in \mathbb{Z},
    \end{cases}
\end{equation}
and finally the normalized Bessel functions of second kind by 
\[
y_\A(x) = 2^\A \Gamma(\A+1) \frac{Y_\A(x)}{x^{\A}}.
\]
These satisfy \eqref{der} as well \cite[eq.~9.1.30]{Abramowitz-Stegun}.

\subsection{Range characterization} We will need the following characterization of the range of SMT from \cite{Agranovsky-Finch-Kuchment-range}.
\begin{theorem}\cite{Agranovsky-Finch-Kuchment-range}\label{T2.4}
    A function $g \in C_c^\infty(\Sn \times [0,2])$ is representable as $g = \Rc f$ for some $f \in C_c^\infty(\Bb)$ if and only if for any $m\geq 0$ and $1 \leq l \leq d_m$, the function $\wh{g}_{m,l}(x, \lambda)$  defined in \eqref{FBT} vanishes at the zeros of the normalized Bessel function $j_{m+\frac{n}{2}-1}(\lambda)$.
\end{theorem}
We note that the moment conditions which were shown to be redundant for the range characterization of SMT for odd dimensions in \cite{Agranovsky-Kuchment-Quinto} were shown to be redundant in all dimensions in the aforementioned paper \cite{Agranovsky-Finch-Kuchment-range}. We will use this characterization to show the sufficiency of our range characterization in the sections below.


\section{Proof of main results}\label{sec:proofs} 
We will prove Theorem \ref{range} first. The proof for the general case follows as a consequence of the radial case. More specifically, we will prove the necessity part of Theorem \ref{range} first. As mentioned above, we need Theorem \ref{Elliptic_Integrals} from which the necessity part would follow immediately. For this reason, we begin with the proof of Theorem \ref{Elliptic_Integrals}. 
\subsection{An identity involving elliptic integrals; proof of Theorem \ref{Elliptic_Integrals}}
We note that for the left hand side integral in Theorem \ref{Elliptic_Integrals}, we use 
    \[
    P(t)=(a-t)(b-t)(c-t)(t-d)
    \]
    and for the right hand side integral, 
    \[
    P(t)=(a-t)(t-b)(t-c)(t-d).
    \]
We require a few preparatory lemmas which we give below.
    We define 
	\begin{align}\label{integrand}
	y(t) = \frac{P(t)}{t},\, t\neq 0.
	\end{align}
	The next lemma follows by a straightforward computation. We skip the proof. 
	\begin{lemma}\label{L2.1}
		The derivative of the function $y$ defined in \eqref{integrand} has the following zeros (and these are the only zeros): 
		\begin{align*}
			r_1 &= \frac{s^2 + u^2 +2}{6} - \frac{v^{1/2}}{6},\\
			r_2 &= \frac{s^2 + u^2 +2}{2} - \frac{w^{1/2}}{2},\\
			r_3 &= \frac{s^2 + u^2 +2}{6} + \frac{v^{1/2}}{6},\\
			r_4 &= \frac{s^2 + u^2 +2}{2} + \frac{w^{1/2}}{2},
			\intertext{where $v= (s^2 + u^2 +2)^2 + 12(1-u^2)(1-s^2)$, $w = 8(u^2 + s^2) + (u^2 - s^2)^2$.}
            \end{align*}
            Furthermore, 
            \[			r_1 < 0 < d \leq r_2 \leq  c \leq r_3 \leq b \leq r_4 \leq a.
		\]
	\end{lemma}
	
    \begin{remark}\label{one-one}
    	An important consequence of the above lemma is that the function $y$ is one-one in the intervals $(d,r_2),(r_2,c),(b,r_4)$ and $(r_4,a)$. This can be seen as follows:\\
    	We have 
    	\begin{align*}
    		y'(t)=\frac{tP'(t)-P(t)}{t^2}= \frac{-3(t-r_1)(t-r_2)(t-r_3)(t-r_4)}{t^2}.
    	\end{align*}
        Since $r_1 < 0 < d \leq r_2 \leq  c \leq r_3 \leq b \leq r_4 \leq a$, 
        $y'> 0$ in $(d,r_2)$ and $(b,r_4)$, and $y'< 0$ in $(r_2,c)$ and $(r_4,a)$. Therefore $y$ is one-one in these four intervals. 
    	\end{remark}
	\begin{lemma}\label{maximum}
    The function $y$ defined in \eqref{integrand} attains the same maximum at $r_2$ and $r_4$, where $r_2$ and $r_4$ are as in Lemma \ref{L2.1}.
	\end{lemma}
	\bpr 
    The fact that $r_2$ and $r_4$ are points of local maxima follows by a straightforward calculation. We skip it. The fact that the maxima are the same follows from the fact that 
    \[
    P(r_4)=4(u^2-s^2)^2 r_4,\quad{and}\quad
    P(r_2)=4(u^2-s^2)^2 r_2.
    \]
     \epr
    As was observed in Remark \ref{one-one}, $y(t)=\frac{P(t)}{t}$ is one-one in $(d,r_2)$, $(r_2,c)$, $(b,r_4)$ and $(r_4,a)$. Also, note that $y\geq 0$ in $(d,c)$ and $(b,a)$, and $y(d)=y(c)=y(b)=y(a)=0$. Using Lemma \ref{maximum}, the maxima attained at $r_2$ and $r_4$ in $(d,c)$ and $(b,a)$, respectively, are the same. Let us call the maximum value $\gamma$. Then, we have that $y$ is a bijective map from each of the intervals $(d,r_2)$, $(r_2,c)$, $(b,r_4)$ and $(r_4,a)$ onto $(0,\gamma)$.

    \begin{lemma}\label{Quartic poly}
    	Let $q \in (0,\gamma)$. Consider the polynomial $Q(t)= -P(t)+ qt$. Let $t_1, t_2, t_3$ and $t_4$ be the roots of $Q$, with $t_1$ and $t_4$ being the largest and the smallest, respectively. Then the roots satisfy the following relation:
    	\[
    	\sqrt{t_1}+\sqrt{t_4} = \sqrt{t_2}+\sqrt{t_3}.
    	\] 
    \end{lemma}
    
    \begin{proof}
    	Let us rewrite $Q$ as
    	\begin{align*}
    	Q(t) &= t^4+ a_3t^3+ a_2t^2+ a_1t+ a_0, \intertext{with}
    	a_0 &= abcd,\\ 
    	a_1-q &= -(abc+abd+acd+bcd),\\
        a_2 &= ab+ac+ad+bc+bd+cd,\\
        a_3 &= -(a+b+c+d).
    	\end{align*}
    	Note that since $q\in (0,\gamma)$, $t_1, t_2, t_3, t_4 \geq d > 0$, where we recall that $d=(1-u)^2$.  
    	We need to show 
    	\begin{align*}
    		\sqrt{t_1}+\sqrt{t_4} = \sqrt{t_2}+\sqrt{t_3}.
            \end{align*}
            Equivalently, 
    \begin{align*} \sqrt{t_1}-\sqrt{t_2} = \sqrt{t_3}-\sqrt{t_4}.
    \end{align*}
    This is true if and only if 
    \begin{align*}
    		(\sqrt{t_1}-\sqrt{t_2})^2 - (\sqrt{t_3}-\sqrt{t_4})^2 = 0.
    	\end{align*}
    	The last but one equivalence  above follows from the fact that $t_1 \geq t_2$ and $t_3 \geq t_4$.
    	Expanding, we get, 
    	\begin{align} \label{relation}
    	(\sqrt{t_1}-\sqrt{t_2})^2 - (\sqrt{t_3}-\sqrt{t_4})^2 = t_1+t_2 - (t_3+t_4) - 2\sqrt{t_1t_2} +2\sqrt{t_3t_4}.
    	\end{align}
        For a quartic polynomial of the form $t^4+ a_3t^3+ a_2t^2+ a_1t+ a_0$, its roots are given as solutions of the following two quadratic equations
        \[
        v^2+ \left[ \frac{a_3}{2} \mp \left( \frac{a_3^2}{4}+ u_1 -a_2 \right)^{1/2} \right]v + \frac{u_1}{2} \mp \left[\left(\frac{u_1}{2}\right)^2 -a_0\right]^{1/2} = 0,
        \]
        where $u_1$ is a real root of the cubic equation 
        \[
        u^3-a_2u^2+ (a_1a_3- 4a_0)u- (a_1^2+ a_0a_3^2- 4a_0a_2) = 0.
        \]
        See \cite[p.17]{Abramowitz-Stegun} for more details.
        Thus we have
        \begin{align*}
        	t_1&= -\frac{a_3}{4} + \frac{R}{2} + \frac{\tilde{D}}{2},\quad 
        	t_2= -\frac{a_3}{4} + \frac{R}{2} - \frac{\tilde{D}}{2},\\
        	t_3&= -\frac{a_3}{4} - \frac{R}{2} + \frac{E}{2},\quad
        	t_4= -\frac{a_3}{4} - \frac{R}{2} - \frac{E}{2},   
        	\intertext{where}
        	R&= \left( \frac{a_3^2}{4}+ u_1 -a_2 \right)^{1/2},\\
    	    \tilde{D}&=	\left[ \frac{a_3^2}{4}+ R^2 -Ra_3- 2u_1+ 4\left(\frac{u_1^2}{4}- a_0 \right)^{1/2} \right]^{1/2},\\
        	E&= \left[ \frac{a_3^2}{4}+ R^2 +Ra_3- 2u_1- 4\left(\frac{u_1^2}{4}- a_0 \right)^{1/2} \right]^{1/2}.
        \end{align*}
        This gives
        \[
        t_1+t_2 - (t_3+t_4) = 2R.
        \]
        Using this in \eqref{relation}, the condition that we need to show is 
        \begin{align*}
        &\sqrt{t_1t_2} -\sqrt{t_3t_4} = R\\
        \Longleftrightarrow& (\sqrt{t_1t_2} -\sqrt{t_3t_4})^2 = R^2.
        \end{align*}
        This again uses the fact that $t_1\geq t_3, t_2\geq t_4$.\\
        We have
        \begin{align*}
        (\sqrt{t_1t_2} -\sqrt{t_3t_4})^2 &= t_1t_2+ t_3t_4-  2\sqrt{t_1t_2t_3t_4}= t_1t_2+ t_3t_4- 2\sqrt{a_0},\\
        t_1t_2+ t_3t_4 &= \frac{a_3^2}{8}+ \frac{R^2}{2}- \frac{\tilde{D}^2+E^2}{4}.
        \end{align*}
        Thus it is reduced to showing
        \[
        a_3^2 - 4R^2 - 2(\tilde{D}^2+E^2) - 16\sqrt{a_0} = 0.
        \]
        Also,
        \[
        2(\tilde{D}^2+E^2) = a_3^2 +4R^2 -8u_1.
        \]
        Then,
        \begin{align*}
        	&a_3^2 - 4R^2 - 2(\tilde{D}^2+E^2) - 16\sqrt{a_0}\\
        	=\,& a_3^2 - 4R^2 - (a_3^2 +4R^2 -8u_1) - 16\sqrt{a_0}\\
        	=\,& -2a_3^2 +8a_2 -16\sqrt{a_0}. 
        \end{align*}
        We finally need to show 
        \[
        a_3^2 -4a_2 +8\sqrt{a_0} = 0.
        \]
        Note that this is independent of $a_1$, which is the only coefficient depending on $q$.\\
        We have
        \begin{align*}
        	&a_3^2 -4a_2 +8\sqrt{a_0}\\
        	=\,& (a+b+c+d)^2 -4(ab+ac+ad+bc+bd+cd) +8\sqrt{abcd}\\
        	=\,& a^2+b^2+c^2+d^2 -2(ab+ac+ad+bc+bd+cd) +8\sqrt{abcd}\\
        	=\,& (a-d)^2 + (b-c)^2 -2(a+d)(b+c) +8\sqrt{abcd}\\
        	=\,& 16(u^2+s^2) -8(1+u^2)(1+s^2) +8(1-u^2)(1-s^2)\\
        	=\,& 0.
        \end{align*}
        This completes the  proof.
    \end{proof}

    \bpr[Proof of Theorem \ref{Elliptic_Integrals}] For $\B=0$, this result is obviously true by straightforward integration. 

    Next for $\B>0$, we first consider:
    	\begin{align}
    		\label{EqI1} I_1&:= \int\limits_d^c \frac{P^{\B}}{t^{\B}\sqrt{t}}\D t\\
           \notag &= 2\int\limits_d^c (\sqrt{t})' \frac{P^{\B}}{t^{\B}}\D t. 
    	\end{align}
    	Integrating by parts, we get,
    	\begin{align*}
    		I_1&= -2\B\int\limits_d^c \sqrt{t} \left(\frac{P}{t}\right)^{\B-1}\left(\frac{P}{t}\right)'\D t\\
    		&= -2\B\int\limits_d^{r_2} \sqrt{t} \left(\frac{P}{t}\right)^{\B-1}\left(\frac{P}{t}\right)'\D t -2\B\int\limits_{r_2}^c \sqrt{t} \left(\frac{P}{t}\right)^{\B-1}\left(\frac{P}{t}\right)'\D t,
    	\end{align*}
        where $r_2$ is as in Lemma \ref{L2.1}.
        We know that $y(t)=\frac{P(t)}{t}$ is a bijective map from $(d,r_2)$ and $(r_2,c)$ onto $(0,\gamma)$. Let $\psi_1$ and $\psi_2$ denote the inverses of $y$ in $(d,r_2)$ and $(r_2,c)$, respectively. Then making the change of variable $y=\frac{P(t)}{t}$, we get,
    	\begin{align*}
    		I_1&= -2\B\int\limits_0^{\gamma} \sqrt{\psi_1} y^{\B-1} \D y + 2\B\int\limits_0^{\gamma} \sqrt{\psi_2} y^{\B-1} \D y\\
    		&= -2\B\int\limits_0^{\gamma} (\sqrt{\psi_1}- \sqrt{\psi_2}) y^{\B-1} \D y.
    	\end{align*} 
    	Similarly, denoting $\psi_3$ and $\psi_4$ as the inverses of $y$ in $(b,r_4)$ and $(r_4,a)$, respectively, we have,
    		\begin{align*}
    		I_2&:= \int\limits_b^a \frac{P^{\B}}{t^{\B}\sqrt{t}}\D t
    		= -2\B\int\limits_0^{\gamma} (\sqrt{\psi_3}- \sqrt{\psi_4}) y^{\B-1} \D y.
    	\end{align*}  
    	Note that $\psi_1 \leq \psi_2 \leq \psi_3 \leq \psi_4$.\\
    	To show $I_1 - I_2 =0$, it is sufficient to show that $\sqrt{\psi_1}- \sqrt{\psi_2}- \sqrt{\psi_3}+ \sqrt{\psi_4} \equiv 0$ in $(0,\gamma)$. More precisely, for any $y_0 \in (0,\gamma)$, we consider the quartic polynomial $P(t)- y_0t$, and we need to show that the sum of square roots of its smallest and largest roots is same as the sum of square roots of its other two roots. This is what we proved in Lemma \ref{Quartic poly}. 
        
        For $\B=-1/2$,  we require a  separate analysis.  In this case, we get a Schwartz-Christoffel integral and the result is true by invoking a suitable fractional linear transformation; see \cite[Page 96]{BF}.  We note that for this case we do not need $a,b,c,d$ to be of the form \eqref{Eqss4.7}. This completes the proof of the theorem for the case $\B=-1/2$ and $\B\geq 0$. 
        
        As stated in the main theorem, the result is, in fact, true for all real $\B>-1$. This can be seen by making a change of variable $\frac{P}{t}=y$ to the integral $I_1$ in \eqref{EqI1} without invoking integration by parts. We skip this proof, since we only require the equality of elliptic integrals for $\B=-1/2$ in $(-1,0)$. 
    	\end{proof}
\subsection{Proof of necessity in Theorem \ref{range}}
We can finally give the proof of necessity in the range characterization results stated above. 
\bpr[Proof of necessity in Theorem \ref{range}]
    	Recall that for $0\leq t\leq 2$ and $\A = \frac{n-2}{2}$, applying Funk-Hecke theorem, we have,  
    	\Beq\label{FH}
    	h(t)=\frac{\o_{n-1}}{\o_{n-2} 2^{n-3} C_m^{\A}(1)}\int\limits_{|1-t|}^{1} uf(u)\Big{\{}\left[(1+u)^2-t^2\right]\left[t^2-(1-u)^2\right]\Big{\}}^{\A-\frac{1}{2}} \D u.
    	\Eeq
    	We need to show that for each $0<s<1$, 
    	\Beq\label{RC}
    	\int\limits_{0}^{1-s} \frac{th(t)}{t^{2\A}} \Big{\{}\left[(1+t)^2-s^2\right]\left[(1-t)^2 - s^2\right]\Big{\}}^{\A- \frac{1}{2}} \D t = \int\limits_{1+s}^{2} \frac{th(t)}{t^{2\A}} \Big{\{}\left[(1+t)^2-s^2\right]\left[(1-t)^2 - s^2\right]\Big{\}}^{\A- \frac{1}{2}} \D t.
    	\Eeq    
    	We note that 
    	\[
    	\left[(1+t)^2-s^2\right]\left[(1-t)^2-s^2\right]=\left[(1-s)^2-t^2\right]\left[(1+s)^2-t^2\right]=\left[t^2-(1-s)^2\right]\left[t^2-(1+s)^2\right].
    	\]
    	Using expression \eqref{FH} for $h$, we can rewrite \eqref{RC} as 
    	\begin{align*}
    		&\int\limits_{0}^{1-s} \int\limits_{1-t}^{1} \frac{tuf(u)}{t^{2\A}} \Big{\{}\left[(1+u)^2-t^2\right]\left[t^2-(1-u)^2\right]\left[(1+s)^2 - t^2 \right]\left[(1-s)^2 - t^2 \right]\Big\} ^{\A- \frac{1}{2}} \D u\, \D t \\
    		=& \int\limits_{1+s}^{2} \int\limits_{t-1}^{1} \frac{tuf(u)}{t^{2\A}} \Big{\{}\left[(1+u)^2-t^2\right]\left[t^2-(1-u)^2\right]\left[(1+s)^2 - t^2 \right] \left[(1-s)^2 - t^2 \right]\Big\} ^{\A- \frac{1}{2}} \D u\, \D t.\\
    	\end{align*}
    	Changing the order of integration, we get,
    	\begin{align*}
    		&\int\limits_{s}^{1} \int\limits_{1-u}^{1-s} \frac{tuf(u)}{t^{2\A}} \Big{\{}\left[(1+u)^2-t^2\right]\left[t^2-(1-u)^2\right]\left[(1+s)^2 - t^2 \right] \left[(1-s)^2 - t^2 \right]\Big\} ^{\A- \frac{1}{2}} \D t \D u \\
    		=& \int\limits_{s}^{1} \int\limits_{1+s}^{1+u} \frac{tuf(u)}{t^{2\A}} \Big{\{}\left[(1+u)^2-t^2\right]\left[t^2-(1-u)^2\right]\left[t^2- (1+s)^2\right] \left[t^2- (1-s)^2 \right]\Big\} ^{\A- \frac{1}{2}} \D t \D u.\\
    	\end{align*} 
    	Thus the necessity part of the theorem would be done if we show that for each $0 < s < u < 1$, 
        \Beq
    	\begin{aligned}
    		\label{LHS-Integral} & \int\limits_{(1-u)^2}^{(1-s)^2} \frac{1}{t^{\A}} \Big{\{}\left[(1+u)^2-t\right]\left[t-(1-u)^2\right]\left[(1+s)^2 - t \right] \left[(1-s)^2 - t \right]\Big\} ^{\A- \frac{1}{2}} \D t \\
    		=&\int\limits_{(1+s)^2}^{(1+u)^2} 
    		\frac{1}{t^{\A}} \Big{\{}\left[(1+u)^2-t\right]\left[t-(1-u)^2\right]\left[t-(1+s)^2 \right] \left[t-(1-s)^2 \right]\Big\} ^{\A- \frac{1}{2}} \D t.
    	\end{aligned}
        \Eeq
        For ease of notation, let us denote
	\begin{align*}
	P(t) = (a-t)(t-b)(t-c)(t-d) = (a-t)(b-t)(c-t)(t-d),
    \end{align*}
	where 
    \Beq\label{Eqss4.7}
    a=(1+u)^2,\, b=(1+s)^2,\, c=(1-s)^2,\, d=(1-u)^2 \mbox{ with }  0<s<u<1.
	\Eeq
    Then \eqref{LHS-Integral} is equivalent to showing that, for each integer $\A\geq 0$, 
    \Beq\label{1.5}
    \int\limits_{d}^{c} \lb \frac{P(t)}{t}\rb^{\A}\frac{1}{\sqrt{P(t)}} \D t  = \int\limits_{b}^{a} \lb \frac{P(t)}{t}\rb^{\A}\frac{1}{\sqrt{P(t)}} \D t. 
    \Eeq
    This is what we proved in Theorem \ref{Elliptic_Integrals} for $\A\geq 0$.
    This concludes the proof of necessity in Theorem \ref{range}.
        \epr
Our next goal is to prove the sufficiency part of Theorem \ref{RC-Radial}. We first derive a new Nicholson-type cross-product identity for Bessel functions. This Nicholson-type identity will be used in the sufficiency part of Theorem \ref{RC-Radial}.

\subsection{A new Nicholson-type cross product formula for Bessel functions; proof of Theorem \ref{Nicholson-type formula}}
\bpr[Proof of Theorem \ref{Nicholson-type formula}] 
Our strategy of proof is similar to the one in \cite{Hrycak-Schmutzhard}. For the sake of completeness, we will give a complete proof. We show by Lemmas \ref{L4.2} and \ref{L4.3} below that  $y(z,w)$, for each fixed $w$, defined in \eqref{Eqss3.1} satisfies the Bessel differential equation,
\[
zy''(z,w)+(2\A+1) y'(z,w)+zy(z,w)=0.
\]
Hence 
\[
y(z,w)= c_1(w) j_{\A}(z)+c_2(w) y_{\A}(z).
\]
We also note that $y(z,w)=-y(w,z)$. Therefore 
\[
c_1(z) j_{\A}(w)+c_2(z) y_{\A}(w)=-c_1(w) j_{\A}(z)-c_2(w) y_{\A}(z)
\]
We choose complex numbers $w_1$ and $w_2$ such that $j_{\A}(w_1) y_{\A}(w_2)-y_{\A}(w_1) j_{\A}(w_2)\neq 0$. Then there exist constants $a_{ij}$ for $1\leq i,j\leq 2$ depending on $w_1, w_2$ and $\A$ such that 
\[
y(z,w)=a_{11} j_{\A}(w) j_{\A}(z)+ a_{21} y_{\A}(w) j_{\A}(z)+ a_{12} j_{\A}(w) y_{\A}(z) + a_{22}y_{\A}(w) y_{\A}(z).
\]
Using the fact that $y(z,w)=-y(w,z)$ combined with the linear independence of $j_{\A}$ and $y_{\A}$, we have that $a_{11}=a_{22}=0$ and $a_{12}=-a_{21}$.  We then have 
\[
a_{12}\lb j_{\A}(w) y_{\A}(z)-y_{\A}(w) j_{\A}(z)\rb=\int\limits_{C} j_{\A}(\sqrt{z^2+w^2-2zw\cosh \zeta}) \sinh^{2\A}\zeta \D \zeta-\sum\limits_{j=0}^{\A-1} a_{j} D_{s}^{j} f(0).
\]
The proof would be complete once we determine $a_{12}$. For this, we first consider the case $\alpha>0$. To determine $a_{12}$ in this case, let us choose $r>0$ and $z=rw$. We then multiply both sides by $r^{\A} w^{2\A}$ and then we let $0<w\to 0$ in both
\[
r^{\A} w^{2\A}\lb j_{\A}(w) y_{\A}(rw)-y_{\A}(w) j_{\A}(rw)\rb 
\]
and 
\[
r^{\A} w^{2\A}\int\limits_{C} j_{\A}(\sqrt{z^2+w^2-2zw\cosh \zeta}) \sinh^{2\A}\zeta \D \zeta-r^{\A} w^{2\A}\sum\limits_{j=0}^{\A-1} a_{j} D_{s}^{j} f(0).
\]
We have, based on \eqref{D_s derivative},   Lemma \ref{Lem4.1} and \eqref{Eqss4.1},
\[
\begin{aligned}
&\lim\limits_{w\to 0}r^{\A} w^{2\A}\Bigg{\{}\int\limits_{0}^{\ln r} j_{\A}(\sqrt{r^2w^2+w^2-2rw^2\cosh \zeta}) \sinh^{2\A}\zeta \D \zeta-\sum\limits_{j=0}^{\A-1} a_{j} D_{s}^{j} f(0)\Bigg{\}}\\
&=\frac{(-1)^{\A} (\A-1)!2^{\A}}{4}{2\A \choose \A}\lb r^{\A}-r^{-\A}\rb.
\end{aligned} 
\]
Next for the case $\A=0$, only the integral expression in \eqref{Eqss3.1} is present and in this case, we have 
\[
\lim\limits_{w\to 0}\int\limits_0^{\ln r} j_{0}(\sqrt{r^2w^2+w^2-2rw^2\cosh \zeta} \D \zeta =j_0(0)\ln r=\ln r.
\]
Similarly we compute the limit, 
\[
\lim\limits_{w\to 0}r^{\A} w^{2\A} \lb j_{\A}(w)y_{\A}(rw)-y_{\A}(w) j_{\A}(rw)\rb.
\]
We note that 
\[
r^{\A} w^{2\A} \lb j_{\A}(w)y_{\A}(rw)-y_{\A}(w) j_{\A}(rw)\rb= C(\A)\lb J_{\A}(w)Y_{\A}(rw)-Y_{\A}(w)J_{\A}(rw)\rb,
\]
where $C(\A)$ is a non-zero constant. We have, from  \cite[Lemma A.2]{Hrycak-Schmutzhard}, 
\[
\lim\limits_{w\to 0}J_{\A}(w)Y_{\A}(rw)-Y_{\A}(w)J_{\A}(rw)=
\begin{cases}
    \frac{1}{\pi \A} \lb r^{\A}-r^{-\A} \rb\mbox{ for } \alpha >0\\
    \frac{2}{\pi} \ln r \mbox{ for } \A=0.
\end{cases}
\]
Hence in both the cases $\A>0$ and $\A=0$, we have that $a_{12}$ is a non-zero constant depending only on $\A$. 
\epr 
\begin{lemma}\label{L4.2}
For $\A\geq 0$, consider the integral, 
\[
\tilde{y}(z,w)=\int\limits_{C} j_{\A}(\sqrt{z^2+w^2-2zw\cosh \zeta}) \sinh^{2\A}(\zeta) \D \zeta.
\]
Then for each fixed $w$ with $\mbox{arg}|w|<\pi$, $\tilde{y}$ satisfies the non-homogeneous ODE: 
\[
\begin{aligned} 
(z\tilde{y}')'(z) +2\A \tilde{y}'(z)+z\tilde{y}(z)=\frac{\A }{2^{2\A-2}w^{2\A}}\lb \frac{z^2-w^2}{z}\rb^{2\A-1}.
\end{aligned}
\]
\end{lemma} 
\bpr Our strategy of proof is similar to the one in \cite{Hrycak-Schmutzhard}. 
We note that $j_{\A}$ satisfies the ODE, 
\[
xj_{\A}''(x)+(2\A+1) j_{\A}'(x)+ xj_{\A}(x)=(x j'_{\A})'(x) +2\A j_{\A}'(x)+ x j_{\A}(x)=0
\]
Define 
\[
u(\xi)=j_{\A}(\sqrt{\xi}).
\]
Then $u$ satisfies the ODE,
\[
4\xi u''(\xi)+4(\A+1)u'(\xi)+u(\xi)=0.
\]
We let $\phi=z^2+w^2-2zw \cosh \zeta$ and $\sinh^{2\A} \zeta = g(\zeta)$ to simplify notation.
We have 
\[
\begin{aligned} 
z\tilde{y}'(z,w)=\int\limits_{C} u'(\phi)(2z^2-2zw\cosh \zeta) g(\zeta) \D \zeta + u(0)g(\log z - \log w).
\end{aligned} 
\]
Differentiating, we get, 
\[
\begin{aligned}
    (z\tilde{y}')'(z,w)&=\int\limits_{C} u''(\phi)(2z-2w \cosh \zeta)(2z^2-2zw \cosh \zeta) g(\zeta) \D \zeta\\
    &+\int\limits_{C} u'(\phi)(4z-2w \cosh\zeta) g(\zeta) \D \zeta +u'(0)(z^2-w^2)g(\log z-\log w)\frac{1}{z}\\
    &+ u(0)g'(\log z - \log w) \frac{1}{z}.
\end{aligned}
\]
This can be rewritten as 
\[
\begin{aligned}
     (z\tilde{y}')'(z,w)&=\int\limits_{C} u''(\phi)(4z^3-8z^2w \cosh \zeta+4zw^2\cosh^2 \zeta)g(\zeta) \D \zeta\\
    &+\int\limits_{C} u'(\phi)(4z-2w \cosh\zeta) g(\zeta) \D \zeta \\
    &+u'(0)(z^2-w^2)g(\log z-\log w)\frac{1}{z}+ u(0)g'(\log z - \log w) \frac{1}{z}.
\end{aligned}
\]
Next 
\[
\begin{aligned} 
(z\tilde{y}')'(z,w) +2\A \tilde{y}'(z,w)&+z\tilde{y}(z,w)=\int\limits_{C} u''(\phi)(4z^3-8z^2w \cosh \zeta+4zw^2\cosh^2 \zeta)g(\zeta) \D \zeta\\
    &+\int\limits_{C} u'(\phi)(4z(1+\A)-2(1+2\A)w \cosh\zeta) g(\zeta) \D \zeta \\
    &+\int\limits_{C} z u(\phi)g(\zeta) \D \zeta+u'(0)(z^2-w^2)g(\log z-\log w)\frac{1}{z}\\
    &+ u(0)g'(\log z - \log w) \frac{1}{z}+2\A u(0)g(\log z-\log w)\frac{1}{z}.
\end{aligned} 
\]
Recall that $u$ satisfies the following ODE: 
\[
4z u''(z)+ 4(\A+1)u'(z)+ u(z) =0.
\]
With this, we have, 
\[
\begin{aligned} 
(z\tilde{y}')'(z,w) +2\A \tilde{y}'(z,w)&+z\tilde{y}(z,w)=\int\limits_{C} u''(\phi)(4zw^2\sinh^2 \zeta)g(\zeta) \D \zeta\\
    &-\int\limits_{C} u'(\phi)(2(1+2\A)w \cosh\zeta) g(\zeta) \D \zeta \\
    &+u'(0)(z^2-w^2)g(\log z-\log w)\frac{1}{z}\\
    &+ u(0)g'(\log z - \log w) \frac{1}{z}+2\A u(0)g(\log z-\log w)\frac{1}{z}.
\end{aligned} 
\]

Considering the integral,
\[
I_1=\int\limits_{C} u''(\phi) (4zw^2\sinh^2 \zeta) g(\zeta) \D \zeta,
\]
and integrating by parts, we have,
\[
I_1=\int\limits_{C} u'(\phi) \lb 2w \sinh \zeta g(\zeta)\rb' \D \zeta  -u'(0)(z^2-w^2) g(\log z -\log w)\frac{1}{z}.
\]
Then we have 
\[
\begin{aligned} 
(z\tilde{y}')'(z,w) +2\A \tilde{y}'(z,w)+z\tilde{y}(z,w)&=\int\limits_{C} u'(\phi)\lb 2w\sinh \zeta)g(\zeta)\rb'  \D \zeta\\
    &-2(1+2\A)\int\limits_{C} u'(\phi)w \cosh\zeta g(\zeta) \D \zeta \\
    & + u(0)g'(\log z - \log w) \frac{1}{z}+2\A u(0)g(\log z-\log w)\frac{1}{z}.
\end{aligned} 
\]
Now substituting $g(\zeta)=\sinh^{2\A} \zeta$, the integrals cancel, and we finally have, noting that, $u(0)=1$, 
\[
\begin{aligned} 
(z\tilde{y}')'(z) +2\A \tilde{y}'(z)+z\tilde{y}(z)&= g'(\log z - \log w) \frac{1}{z}+2\A g(\log z-\log w)\frac{1}{z}\\
&=\frac{\A}{2^{2\A-2}w^{2\A}}\lb \frac{z^2-w^2}{z}\rb^{2\A-1}.
\end{aligned}
\]
\epr 
Next let us define
\[
f(s)=\frac{\A \mbox{sgn}(|z|-|w|)}{2^{2\A-2}}\frac{\Bigg{\{}(z^2-w^2)^2-2s^2(z^2+w^2)+s^4\Bigg{\}}^{\frac{2\A-1}{2}}}{z^{2\A}w^{2\A}},
\]
and let 
\[
P(z,w)=\sum\limits_{j=0}^{\A-1} a_j D_s^{j}f(0).
\]
\begin{lemma}\label{L4.3}
    There exist constants $\{a_j\}$ such that $P(z,w)$ as a function of $z$ for $z,w\neq 0$ and $z\neq w$, solves 
    \[
    (zy')'(z,w) +2\A y'(z,w)+zy(z,w)=\frac{\A }{2^{2\A-2}w^{2\A}}\lb \frac{z^2-w^2}{z}\rb^{2\A-1}.
    \]
    In fact, the constants $a_j$ are recursively defined by $a_j=2(\A-j)a_{j-1}$ with $a_0=1$. 
Consequently, 
\[
y(z,w)=\tilde{y}(z,w)-P(z,w)=\int\limits_{C} j_{\A}(\sqrt{z^2+w^2-2zw\cosh \zeta}) \sinh^{2\A}\zeta  \D \zeta-\sum\limits_{j=0}^{\A-1} a_j D_s^{j}f(0)
\]
solves 
\[
(zy')'(z,w) +2\A y'(z,w)+zy(z,w)=0.
\]
\end{lemma}
\begin{proof} We are interested in taking $D$ derivatives of $f(s)$ up to order $\A-1$. For simplicity, let us denote $f(s)= \frac{\A}{2^{2\A-2}}\tilde{f}(s)$  with 
        \[
        \tilde{f}(s)= \frac{\mbox{sgn}(|z|-|w|)\Bigg{\{}(z^2-w^2)^2-2s^2(z^2+w^2)+s^4\Bigg{\}}^{\frac{2\A-1}{2}}}{z^{2\A}w^{2\A}}
        \]and $\tilde{P}(z,w)=\sum\limits_{j=0}^{\A-1} a_j D_s^{j}\tilde{f}(0)$.\\ 
		We define
		\begin{align*}
			Q(s) &= (z^2-w^2)^2-2s^2(z^2+w^2)+s^4,\\
			R(s) &= Q(s)^{\frac{2\A-1}{2}}.
		\end{align*}
		Then we would like to take $D$ derivatives  up to order $\A-1$ of $R(s)$. To find explicit form of these $D$ derivatives, we use the Fa\`a di Bruno's formula; see Lemma \ref{L2.2}. We set 
        Indeed, in our situation, 
		\[
		R(s)=x^{\frac{2\A-1}{2}} \mbox{ and } 
		Q(s)=(z^2-w^2)^2-2s^2(z^2+w^2)+s^4.
		\]
		In our current set-up, we have $D^{j}Q(s)=0$ for $j\geq 3$. We also have 
		\[
		DQ(s)=4(s^2-(z^2+w^2)),\quad
		D^2 Q(s)=8.
		\]
		Evaluating these at $s=0$ we get,
		\[
		Q(0)=(z^2-w^2)^2,\quad DQ(0)=-4(z^2+w^2),\quad D^2Q(0)=8.
		\]
		Due to the existence of only two non-trivial $D$ derivatives of $Q(s)$, recall from Lemma \ref{L2.3} that the Bell polynomials are subject to the following two conditions: 
		\[
		j_1+j_2=q, \quad j_1+2j_2=j.
		\]
		Solving this gives $j_2=j-q$ and $j_1=2q-j$ with $q\geq j/2$. We now arrive at the following formula for $D^{j} [R(Q(s))]$ evaluated at $s=0$: 
        \Beq\label{D_s derivative}
		\begin{aligned}
			D^{j}[R(Q(s))]|_{s=0} &= \sum\limits_{q\geq j/2}^{j} \frac{(2\A)! (\A-q)!}{4^q\A!(2\A-2q)!} (Q(0))^{\A-q-1/2} \frac{j!}{(2q-j)!(j-q)!} (DQ(0))^{2q-j} \left(\frac{D^2Q(0)}{2}\right)^{j-q}\\
			&= \sum\limits_{q\geq j/2}^{j} \frac{(2\A)! (\A-q)!}{4^q\A!(2\A-2q)!}  \frac{(-1)^j j! 4^q}{(2q-j)!(j-q)!} \mathrm{sgn}(|z|-|w|)(z^2-w^2)^{2\A-2q-1} (z^2+w^2)^{2q-j}\\
			&= \sum\limits_{q\geq j/2}^{j} \frac{(-1)^j (2\A)! (\A-q)! j!}{(2\A-2q)! \A! (2q-j)!(j-q)!}\mathrm{sgn}(|z|-|w|) (z^2-w^2)^{2\A-2q-1} (z^2-w^2+2w^2)^{2q-j}\\
			&= \sum\limits_{q\geq j/2}^{j} \sum\limits_{r=0}^{2q-j} {2q-j \choose r} \frac{(-1)^j (2\A)! (\A-q)! j!}{(2\A-2q)! \A! (2q-j)!(j-q)!} \mathrm{sgn}(|z|-|w|)(z^2-w^2)^{2\A-j-r-1} (2w^2)^{r}.	
		\end{aligned}
        \Eeq
		Then, 
		\Beq\label{Eqss4.11}
			D^j_s \tilde{f}(0) = \sum\limits_{q\geq j/2}^{j} \sum\limits_{r=0}^{2q-j} {2q-j \choose r} \frac{(-1)^j (2\A)! (\A-q)! j!}{(2\A-2q)! \A! (2q-j)!(j-q)!}  (2w^2)^{r}\frac{(z^2-w^2)^{2\A-j-r-1}}{z^{2\A}w^{2\A}}.	
		\Eeq 
        We note that $\mathrm{sgn}(|z|-|w|)^2$ appears in the expression for $D^{j}\tilde{f}(0)$ and consequently, this term can be ignored going forward. 
        
		Let us write \eqref{Eqss4.11} as 
		\[
			D_s^j \tilde{f}(0) = \sum\limits_{q\geq j/2}^{j} \sum\limits_{r=0}^{2q-j} C(j,q,r) \frac{(z^2-w^2)^{2\A-j-r-1}}{z^{2\A}w^{2\A}},
            \]
			where 
            \[
			C(j,q,r) = {2q-j \choose r} \frac{(-1)^j (2\A)! (\A-q)! j!}{(2\A-2q)! \A! (2q-j)!(j-q)!}  (2w^2)^{r}.
		\]
		Define
		\[
		y_0 = \frac{(z^2-w^2)^{2\A-j-r-1}}{z^{2\A}}.
		\]
		Let us denote $\Lc (y)=(zy')'(z,w) +2\A y'(z,w)+zy(z,w)$. We next compute $\Lc (y_0)$.
		We have 
		\begin{align*}
			y_0' &= 2(2\A-j-r-1)\frac{(z^2-w^2)^{2\A-j-r-2}}{z^{2\A-1}} - 2\A\frac{(z^2-w^2)^{2\A-j-r-1}}{z^{2\A+1}}\\
			&= 2\frac{(z^2-w^2)^{2\A-j-r-2}}{z^{2\A+1}} (z^2(\A-j-r-1) + \A w^2),\\
			\intertext{and}
			y_0'' &= 4(\A-j-r-1)(2\A-j-r-2)\frac{(z^2-w^2)^{2\A-j-r-3}}{z^{2\A-2}} -2(2\A-1)(\A-j-r-1)\frac{(z^2-w^2)^{2\A-j-r-2}}{z^{2\A}}\\
			&+ 4\A w^2(2\A-j-r-2) \frac{(z^2-w^2)^{2\A-j-r-3}}{z^{2\A}} -(2\A)(2\A+1)w^2\frac{(z^2-w^2)^{2\A-j-r-2}}{z^{2\A+2}}.
		\end{align*}
		Then 
		\begin{align*}
			\Lc (y_0) &= zy_0'' + (2\A+1)y_0' +zy_0\\
			&= 4(\A-j-r-1) \frac{(z^2-w^2)^{2\A-j-r-2}}{z^{2\A-1}} + \frac{(z^2-w^2)^{2\A-j-r-1}}{z^{2\A-1}}\\ &+4(2\A-j-r-2)\frac{(z^2-w^2)^{2\A-j-r-3}}{z^{2\A-1}}(z^2(\A-j-r-1) +\A w^2)\\
			&= 4(\A-j-r-1)(2\A-j-r-1) \frac{(z^2-w^2)^{2\A-j-r-2}}{z^{2\A-1}} + \frac{(z^2-w^2)^{2\A-j-r-1}}{z^{2\A-1}}\\ &+4(2\A-j-r-1)(2\A-j-r-2)w^2\frac{(z^2-w^2)^{2\A-j-r-3}}{z^{2\A-1}}.
		\end{align*}
		Therefore 
		\begin{align*}
			\Lc (\tilde{P}) = \sum\limits_{j=0}^{\A-1} \sum\limits_{q\geq j/2}^{j} \sum\limits_{r=0}^{2q-j} a_j C(j,q,r) \Big{\{} 
			&\frac{(z^2-w^2)^{2\A-j-r-1}}{z^{2\A-1}w^{2\A}} +4(\A-j-r-1)(2\A-j-r-1) \frac{(z^2-w^2)^{2\A-j-r-2}}{z^{2\A-1}w^{2\A}}\\ &+4(2\A-j-r-1)(2\A-j-r-2)w^2\frac{(z^2-w^2)^{2\A-j-r-3}}{z^{2\A-1}w^{2\A}} \Big{\}}.
		\end{align*}
        We show that with the choice of constants, $a_j=2(\A-j)a_{j-1}$, $a_0=1$, we have that $\Lc(\tilde{P})=\frac{(z^2-w^2)^{2\A-1}}{z^{2\A-1}w^{2\A}}$.
		Since the denominators are same in both the expressions, we can simply compare the coefficients of different powers of $(z^2-w^2)$ to find $a_j$.\\
		Let us set $j+r=l$. We consider
		\begin{align*}
			S:= \sum\limits_{j=0}^{\A-1} \sum\limits_{q\geq j/2}^{j} \sum\limits_{l=j}^{2q} a_j K(j,q,l) \Big{\{} 
			&(z^2-w^2)^{2\A-l-1} + 4(\A-l-1)(2\A-l-1)(z^2-w^2)^{2\A-l-2}\\ &+4(2\A-l-1)(2\A-l-2)w^2(z^2-w^2)^{2\A-l-3} \Big{\}},\\
			\intertext{with}
			K(j,q,l)= C(j,q,l-j)&= {2q-j \choose l-j} \frac{(-1)^j (2\A)! (\A-q)! j!}{(2\A-2q)! \A! (2q-j)!(j-q)!}  (2w^2)^{l-j}.
		\end{align*}
		Interchanging the order of summation, we have
		\begin{align*}
			S= \sum\limits_{l=0}^{2\A-2} \sum\limits_{j\geq l/2}^{\min\{l,\A-1\}} \sum\limits_{q\geq l/2}^{j} a_j K(j,q,l) \Big{\{} 
			&(z^2-w^2)^{2\A-l-1} + 4(\A-l-1)(2\A-l-1)(z^2-w^2)^{2\A-l-2}\\ &+4(2\A-l-1)(2\A-l-2)w^2(z^2-w^2)^{2\A-l-3} \Big{\}}.
		\end{align*}
		We set the coefficient of $(z^2-w^2)^{2\A-1}$ to $1$.
		The only term having $(z^2-w^2)^{2\A-1}$ is given by the first summand with $l=0$. For $l=0$, the only possible triple for $(j,q,r)$ is $(0,0,0)$, since $q\leq j \leq l$ and $r \leq 2q-j$. Since $C(0,0,0)=1$, we have $a_0 = 1$.
		
		Next, we consider the coefficient of $(z^2-w^2)^{2\A-2}$. This is contributed by two terms: the first summand with $l=1$ and the second summand with $l=0$. The only triple $(j,q,r)$ corresponding to $l=1$ is $(1,1,0)$. Thus, we have
		\[
			4a_0C(0,0,0)(\A-1)(2\A-1) + a_1C(1,1,0) = 0.
            \]
            Hence 
			\[
            4(\A-1)(2\A-1) -2(2\A-1)a_1 = 0. 
		\]
		Since $\A$ is an integer, $2\A-1\neq 0$ and we get $a_1=2(\A-1)$.\\

        Let us now consider the terms involving $(z^2-w^2)^{2\A-1-m}$ for $2\leq m \leq 2\A-1$. Such terms are given by first, second and third summands with $l=m, l=m-1$ and $l=m-2$, respectively. 
		Then the coefficient of $(z^2-w^2)^{2\A-1-m}$ is given as
		\begin{align*}
			S(m)&:= \sum\limits_{j\geq m/2}^{\min\{m,\A-1\}} \sum\limits_{q\geq m/2}^{j} a_j K(j,q,m)\\
			&+\sum\limits_{j\geq (m-1)/2}^{\min\{m-1,\A-1\}} \sum\limits_{q\geq (m-1)/2}^{j} 4 a_j K(j,q,m-1) (\A-m)(2\A-m)\\ 
			&+\sum\limits_{j\geq (m-2)/2}^{\min\{m-2,\A-1\}} \sum\limits_{q\geq (m-2)/2}^{j} 4 a_j K(j,q,m-2)(2\A-m)(2\A-m+1)w^2.
		\end{align*}
		Re-indexing the second and third summations, we get
		\begin{align*}
			S(m)&= \sum\limits_{j\geq m/2}^{\min\{m,\A-1\}} \sum\limits_{q\geq m/2}^{j} a_j C(j,q,m-j)\\
			&+\sum\limits_{j\geq (m+1)/2}^{\min\{m,\A\}} \sum\limits_{q\geq (m-1)/2}^{j-1} 4 a_{j-1} C(j-1,q,m-j) (\A-m)(2\A-m)\\ 
			&+\sum\limits_{j\geq m/2}^{\min\{m-1,\A\}} \sum\limits_{q\geq (m-2)/2}^{j-1} 2(2w^2) a_{j-1} C(j-1,q,m-j-1)(2\A-m)(2\A-m+1).
		\end{align*}
		Recall that 
		\[
		C(j,q,r) = {2q-j \choose r} \frac{(-1)^j (2\A)! (\A-q)! j!}{(2\A-2q)! \A! (2q-j)!(j-q)!}  (2w^2)^{r}.
		\]
		Let us write $C(j,q,r)=(w^2)^r \tilde{C}(j,q,r)$.\\
		We equate the coefficients of the polynomial in $w^2$ with $0$ to find $a_j$. Let us fix $j=p$.\\
		Then the coefficient of $(w^2)^{m-p}$ is given as:
		\begin{align*}
			S(m,p)&:= \sum\limits_{q\geq m/2}^{p} a_p \tilde{C}(p,q,m-p)\\
			&+\sum\limits_{q\geq (m-1)/2}^{p-1} 4 a_{p-1} \tilde{C}(p-1,q,m-p) (\A-m)(2\A-m)\\ 
			&+\sum\limits_{q\geq (m-2)/2}^{p-1} 4 a_{p-1} \tilde{C}(p-1,q,m-p-1)(2\A-m)(2\A-m+1).
		\end{align*}
		Re-indexing the second and third summations, we get
		\begin{align*}
			S(m,p)&= \sum\limits_{q\geq m/2}^{p} a_p \tilde{C}(p,q,m-p)\\
			&+\sum\limits_{q\geq (m+1)/2}^{p} 4 a_{p-1} \tilde{C}(p-1,q-1,m-p) (\A-m)(2\A-m)\\ 
			&+\sum\limits_{q\geq m/2}^{p} 4 a_{p-1} \tilde{C}(p-1,q-1,m-p-1)(2\A-m)(2\A-m+1).
		\end{align*}
		Let us assume $m$ is odd, as the other case follows similarly. Then substituting the expression for $\tilde{C}(j,q,r)$, we have 
		\begin{align*}
			S(m,p)&= \frac{(-1)^{p}2^{m-p}(p-1)!(2\A)!}{\A!(m-p)!}\sum\limits_{q= (m+1)/2}^{p}\Bigg{\{} \frac{p a_p (\A-q)!}{(2\A-2q)!(2q-m)!(p-q)!}\\
			&- \frac{4a_{p-1}(\A-m)(2\A-m)(\A-q+1)!}{(2\A-2q+2)!(2q-m-1)!(p-q)!}\\
			&- \frac{2a_{p-1}(2\A-m)(2\A-m+1)(m-p)(\A-q+1)!}{(2\A-2q+2)!(2q-m)!(p-q)!} \Bigg{\}}.
		\end{align*} 
		Since we set $S(m,p)=0$, we ignore the constant outside the summation from now on. Then we have 
		\begin{align*}
			S(m,p)&= \sum\limits_{q= (m+1)/2}^{p} \Bigg{\{} \frac{pa_p(\A-p)!}{(2\A-m)!}{\A-q\choose p-q}{2\A-m \choose 2q-m}\\
			&-\frac{2a_{p-1}(\A-m)(2\A-m)(\A-p)!}{(2\A-m)!}{\A-q\choose p-q}{2\A-m \choose 2q-m-1}\\
			&-\frac{a_{p-1}(2\A-m)(m-p)(\A-p)!}{(2\A-m)!}{\A-q\choose p-q}{2\A-m+1 \choose 2q-m} \Bigg{\}}.
		\end{align*}
		We ignore the constant $\frac{(\A-p)!}{(2\A-m)!}$ and rewrite it as
		\begin{align*}
			S(m,p)&= \sum\limits_{q= (m+1)/2}^{p} {\A-q\choose p-q} \Bigg{\{} pa_p {2\A-m \choose 2q-m}- 2a_{p-1}(\A-m)(2\A-m){2\A-m \choose 2q-m-1}\\
            &- a_{p-1}(2\A-m)(m-p){2\A-m+1 \choose 2q-m} \Bigg{\}}. 
		\end{align*}
		Using Pascal's rule: ${n+1 \choose r}= {n \choose r}+{n \choose r-1}$ in the second summand, we have
		\begin{align*}
			S(m,p)&= \sum\limits_{q= (m+1)/2}^{p} {\A-q\choose p-q} \Bigg{\{} {2\A-m \choose 2q-m} (pa_p+ 2a_{p-1}(\A-m)(2\A-m))\\
			&- {2\A-m+1 \choose 2q-m}( a_{p-1}(2\A-m)(2\A-m-p)) \Bigg{\}}.
		\end{align*}
		Let us show that 
			\[
			C:=\sum\limits_{q=\frac{m+1}{2}}^{p}{\A-q \choose p-q}{2\A -m \choose 2q-m}=2^{2p-m}{\A+p-m-1\choose 2p-m}+2^{2p-m-1}{\A+p-m-1\choose 2p-m-1}
			\]
			Let us first observe that we can replace the lower limit by $0$. We have 
			\[
			\begin{aligned} 
				\sum\limits_{q=\frac{m+1}{2}}^{p}{\A-q \choose p-q}{2\A -m \choose 2q-m}&=\sum\limits_{q=0}^{p}{\A-q \choose p-q}{2\A -m \choose 2q-m}\\
				&=\sum\limits_{q=0}^{p}{\A -p+q\choose q}{2\A -m \choose 2p-m-2q}\\
				&=\frac{1}{2\pi \I}\int\limits_{|z|=\ve}\sum\limits_{q=0}^{p}{\A -p+q\choose q}\frac{(1+z)^{2\A-m}}{z^{2p-m-2q+1}}\D z. 
			\end{aligned} 
			\]
			The upper limit can be considered as $\infty$ and we have 
			\[
			\begin{aligned}
				C&=\frac{1}{2\pi \I}\int\limits_{|z|=\ve}\sum\limits_{q=0}^{\infty}{\A -p+q\choose q}\frac{(1+z)^{2\A-m}}{z^{2p-m-2q+1}}\\
				&=\frac{1}{2\pi \I}\int\limits_{|z|=\ve}\frac{(1+z)^{2\A-m}}{z^{2p-m+1}}\frac{1}{(1-z^2)^{\A-p+1}} \D z\\
				&=\frac{1}{2\pi \I}\int\limits_{|z|=\ve}\frac{(1+z)^{\A+p-m-1}}{z^{2p-m+1}}\frac{1}{(1-z)^{\A-p+1}} \D z\\
				&=\frac{1}{2\pi \I}\int\limits_{|z|=\ve}\sum\limits_{j=0}^{\A+p-m-1}\sum\limits_{l=0}^{\infty}{\A+p-m-1\choose j} {\A-p+l\choose l} \frac{z^{j+l}}{z^{2p-m+1}} \D z\\
				&=\sum\limits_{j=0}^{\A+p-m-1}{\A+p-m-1\choose j}{\A+p-m-j\choose 2p-m-j}\\
				&=\frac{1}{2\pi \I}\int\limits_{|z|=\ve}\frac{(1+z)^{\A+p-m}}{z^{2p-m+1}} \sum\limits_{j=0}^{\A+p-m-1}{\A+p-m-1\choose j}\lb \frac{z}{1+z}\rb^{j} \D z\\
				&=\frac{1}{2\pi \I}\int\limits_{|z|=\ve}\frac{(1+z)^{\A+p-m}}{z^{2p-m+1}}\frac{(1+2z)^{\A+p-m-1}}{(1+z)^{\A+p-m-1}} \D z\\
				&=\frac{1}{2\pi \I}\int\limits_{|z|=\ve}\frac{(1+z)(1+2z)^{\A+p-m-1}}{z^{2p-m+1}} \D z\\
				&=2^{2p-m}{\A+p-m-1\choose 2p-m}+2^{2p-m-1}{\A+p-m-1\choose 2p-m-1}.
			\end{aligned}
			\]
            The same argument gives
			\[
			D:=\sum\limits_{q=\frac{m+1}{2}}^{p}{\A-q \choose p-q}{2\A -m +1\choose 2q-m}=2^{2p-m}{\A+p-m\choose 2p-m}.
			\]
			Let us now show that for $a_p=2a_{p-1}(\A-p)$, $S(m,p)$ given below is $0$.
			That is, we want to show that 
			\begin{align*}
				S(m,p)&= \sum\limits_{q= (m+1)/2}^{p} {\A-q\choose p-q} \Bigg{\{} {2\A-m \choose 2q-m} (2p(\A-p)+ 2(\A-m)(2\A-m))\\
				&- {2\A-m+1 \choose 2q-m}((2\A-m)(2\A-m-p)) \Bigg{\}}=0.
			\end{align*}
			We have 
			\[
			\begin{aligned}
				S(m,p)&= \Bigg{\{}2^{2p-m-1}{\A+p-m-1\choose 2p-m-1}+2^{2p-m}{\A+p-m-1\choose 2p-m}\Bigg{\}}(2p(\A-p)+2(\A-m)(2\A-m))\\
				&-2^{2p-m}{\A+p-m\choose 2p-m}((2\A-m)(2\A-m-p))\\
				&=2^{2p-m}\Bigg{\{}\Big{\{}{\A+p-m-1\choose 2p-m-1}+2{\A+p-m-1\choose 2p-m}\Big{\}}(p(\A-p)+(\A-m)(2\A-m))\\
				&-{\A+p-m\choose 2p-m}((2\A-m)(2\A-m-p))\Bigg{\}}\\
				&=2^{2p-m}\Bigg{\{}\Big{\{}{\A+p-m\choose 2p-m}+{\A+p-m-1\choose 2p-m}\Big{\}}(p(\A-p)+(\A-m)(2\A-m))\\
				&-{\A+p-m\choose 2p-m}((2\A-m)(2\A-m-p))\Bigg{\}}.
			\end{aligned}
			\]
			In the last step, we use Pascal's rule. Further simplifying, we have 
			\[
			\begin{aligned}
				S(m,p)&=2^{2p-m}\Bigg{\{}-{\A+p-m \choose 2p-m}(\A-p)(2\A-m-p)+{\A+p-m-1\choose 2p-m}(p(\A-p)+(\A-m)(2\A-m))\Bigg{\}}\\
				&=2^{2p-m}\Bigg{\{}{\A+p-m-1\choose 2p-m}\Bigg{\{}(p(\A-p)+(\A-m)(2\A-m))-(\A+p-m)(2\A-m-p)\Bigg{\}}\\
				&=0.
			\end{aligned}
			\]
			Hence we have shown that $S(m,p)=0.$ This shows that with $a_j=2(\A-j)a_{j-1}$, $P(z,w)$ as a function of $z$ solves 
			\[
			(zy')'(z,w) +2\A y'(z,w)+zy(z,w)=\frac{\A}{2^{2\A-2}w^{2\A}}\lb \frac{z^2-w^2}{z}\rb^{2\A-1}.
			\]
            Finally, $y(z,w) = \tilde{y}(z,w)-P(z,w)$ solves 
            \[
(zy')'(z,w) +2\A y'(z,w)+zy(z,w)=0.
\]
This completes the proof. 
	\end{proof}
    The following lemma was used in the proof of Theorem \ref{Nicholson-type formula}.
    \begin{lemma}\label{Lem4.1} Let $\A $ be a positive integer. We have the following equality: 
        \[
        \begin{aligned} 
    S:=\frac{(-1)^{\A-1}}{\A} \sum\limits_{q=0}^{\A-1} \frac{(2\A)! (\A-q)!}{(2\A-2q)!}  &\frac{1}{(2q-\A+1)!(\A-1-q)!} A^{2\A-2q-1} B^{2q-\A+1} \\&=\frac{(-1)^{\A-1} (\A-1)!}{4}{2\A \choose \A}\lb (A+B)^{\A}-(B-A)^{\A}\rb.
    \end{aligned} 
        \]
    \end{lemma}
    \bpr
    We have 
    \[
    S=\frac{(-1)^{\A-1}(\A-1)!}{2}{2\A\choose \A} \sum\limits_{q=0}^{\A-1}{\A \choose 2q+1-\A} A^{2\A-2q-1} B^{2q-\A+1}.
    \]
    We replace the index $q$ by $\A-1-q$. We then get,
    \[
    \begin{aligned} 
    S&=\frac{(-1)^{\A-1}(\A-1)!}{2}{2\A\choose \A} \sum\limits_{q=0}^{\A-1}{\A \choose \A-2q-1}A^{2q+1} B^{\A-2q-1}\\
    &=\frac{(-1)^{\A-1}(\A-1)!}{2}{2\A\choose \A} \sum\limits_{q=0}^{\A}{\A \choose 2q+1}A^{2q+1} B^{\A-2q-1}
    \end{aligned} 
    \]
    We also note that we have increased the upper index of the last sum to $\A$ and this does not affect the sum. 

    We have 
    \[
    \begin{aligned} 
    & (A+B)^{\A}=\sum\limits_{r=0}^{\A} {\A \choose r} A^r B^{\A-r}\\
    &(B-A)^{\A}=\sum\limits_{r=0}^{\A} {\A \choose r} (-1)^r A^r B^{\A-r}.
    \end{aligned} 
    \]
    We then have 
    \[
    \begin{aligned} 
    (A+B)^{\A}-(B-A)^{\A}&=\sum\limits_{r=0}^{\A} {\A \choose r} A^r B^{\A-r}(1-(-1)^r)\\
    &=2\sum\limits_{r=0, r\mathrm{-odd }}^{\A}{\A \choose r} A^r B^{\A-r}\\
    &=2\sum\limits_{q=0}^{\A} {\A \choose 2q+1} A^{2q+1} B^{\A-2q-1}.
    \end{aligned} 
    \]
    Going back, we get,
    \[
    S=\frac{(-1)^{\A-1} (\A-1)!}{4}{2\A \choose \A}\lb (A+B)^{\A}-(B-A)^{\A}\rb.
    \]
    \epr
    Using $A=z^2-w^2$ and $B=z^2+w^2$, we then have 
    \Beq\label{Eqss4.1}
    \begin{aligned} 
    &\frac{(-1)^{\A-1}}{\A} \sum\limits_{q=0}^{\A-1} \frac{(2\A)! (\A-q)!}{(2\A-2q)!}  \frac{1}{(2q-\A+1)!(\A-1-q)!} (z^2-w^2)^{2\A-2q-1} (z^2+w^2)^{2q-\A+1}\\
    &=\frac{(-1)^{\A-1} (\A-1)!2^{\A}}{4}{2\A \choose \A}\lb z^{2\A}-w^{2\A}\rb.
    \end{aligned} 
    \Eeq
    \subsection{A cross product identity for Bessel functions; Proof of Theorem \ref{CP-Identity}}

    \begin{proof}[Proof of Theorem \ref{CP-Identity}] We will show that, for $\lambda>0$, 
	\[
	I := \int\limits_{0}^{2} th(t) \{j_{\A}(\lambda t)y_{\A}(\lambda) - j_{\A}(\lambda)y_{\A}(\lambda t)\} \D t = 0.
	\]
	Using the formula from Theorem \ref{Nicholson-type formula} with $z=\lambda t$, $w=\lambda$, we have 
	\begin{align*}
		&C \Big{\{}j_{\A}(\lambda t)y_{\A}(\lambda) - j_{\A}(\lambda)y_{\A}(\lambda t)\Big{\}}
		=\int\limits_{C} j_{\A}(\lambda\sqrt{1+t^2-2t\cosh \zeta}) \sinh^{2\A}\zeta \D \zeta-\sum\limits_{j=0}^{\A-1} a_{j} D_{s}^{j} f(0),
        \end{align*}
		where we recall that 
        \begin{align*}
		f(s)&=\frac{\A \,\mathrm{sgn}(\lambda(t-1))}{2^{2\A-2}}\frac{\Bigg{\{}\lambda^4(t^2-1)^2-2\lambda^2 s^2(1+t^2)+s^4\Bigg{\}}^{\frac{2\A-1}{2}}}{{\lambda}^{4\A}t^{2\A}}\\
		&=\frac{\A \,\mathrm{sgn}(\lambda(t-1))}{2^{2\A-2}}\frac{\Big{\{} \left[  \lambda^2(1+t)^2-s^2\right]\left[\lambda^2(1-t)^2 - s^2\right]\Big{\}}^{\frac{2\A-1}{2}}}{{\lambda}^{4\A}t^{2\A}}.
	\end{align*}
	For $0<t<1$, since $\log t<0$, we have
	\begin{align*}
	I_1&:=\int\limits_{C} j_{\A}(\lambda\sqrt{1+t^2-2t\cosh \zeta}) \sinh^{2\A}\zeta \D \zeta=\int\limits_{\log t}^{0} j_{\A}(\lambda\sqrt{1+t^2-2t\cosh \zeta}) \sinh^{2\A}\zeta \D \zeta. 
	\end{align*}
	Substituting $s=\sqrt{1+t^2-2t\cosh \zeta}$, we get,
	\begin{align*}
		I_1&=-\int\limits_{0}^{1-t} \frac{s j_{\A}(\lambda s)\Big{\{} \left[(1+t)^2-s^2\right]\left[(1-t)^2 - s^2\right]\Big{\}}^{\A-\frac{1}{2}}}{2^{2\A-1}t^{2\A}} \D s.
	\end{align*}
	Similarly, for $1<t<2$, since $\log t>0$, we have
	\begin{align*}
		I_2:=& \int\limits_{C} j_{\A}(\lambda\sqrt{1+t^2-2t\cosh \zeta}) \sinh^{2\A}\zeta \D \zeta
		=\int\limits_{0}^{\log t} j_{\A}(\lambda\sqrt{1+t^2-2t\cosh \zeta}) \sinh^{2\A}\zeta \D \zeta\\
		=&\int\limits_{0}^{t-1} \frac{s j_{\A}(\lambda s)\Big{\{} \left[(1+t)^2-s^2\right]\left[(1-t)^2 - s^2\right]\Big{\}}^{\A-\frac{1}{2}}}{2^{2\A-1}t^{2\A}} \D s.
	\end{align*}
	Using these, we have,
	\begin{align*}
		I&= \int\limits_{0}^{2} th(t) \{j_{\A}(\lambda t)y_{\A}(\lambda) - j_{\A}(\lambda)y_{\A}(\lambda t)\} \D t\\
		&=C_1\int\limits_{0}^{2} th(t) \mbox{sgn}(t-1) \int\limits_{0}^{|1-t|} \frac{s j_{\A}(\lambda s)}{t^{2\A}} \Big{\{} \left[(1+t)^2-s^2\right]\left[(1-t)^2 - s^2\right]\Big{\}}^{\A-\frac{1}{2}} \D s \D t\\
		&-C_2\int\limits_{0}^{2} th(t)\sum\limits_{j=0}^{\A-1} a_{j}\,\mathrm{sgn}(\lambda(t-1)) D_{s}^{j}\Bigg{\{}\frac{\Big{\{} \left[\lambda^2(1+t)^2-s^2\right]\left[\lambda^2(1-t)^2 - s^2\right]\Big{\}}^{\A-\frac{1}{2}}}{{\lambda}^{4\A}t^{2\A}}\Bigg{\}}_{|_{s=0}} \D t\\
		&=-C_1\int\limits_{0}^{1} th(t) \int\limits_{0}^{1-t} \frac{s j_{\A}(\lambda s)}{t^{2\A}} \Big{\{} \left[(1+t)^2-s^2\right]\left[(1-t)^2 - s^2\right]\Big{\}}^{\A-\frac{1}{2}} \D s \D t\\
		&+C_1\int\limits_{1}^{2} th(t) \int\limits_{0}^{t-1} \frac{s j_{\A}(\lambda s)}{t^{2\A}} \Big{\{} \left[(1+t)^2-s^2\right]\left[(1-t)^2 - s^2\right]\Big{\}}^{\A-\frac{1}{2}} \D s \D t\\
		&+C_2\int\limits_{0}^{1} th(t)\sum\limits_{j=0}^{\A-1} a_{j} D_{s}^{j}\Bigg{\{}\frac{\Big{\{} \left[\lambda^2(1+t)^2-s^2\right]\left[\lambda^2(1-t)^2 - s^2\right]\Big{\}}^{\A-\frac{1}{2}}}{{\lambda}^{4\A}t^{2\A}}\Bigg{\}}_{|_{s=0}} \D t\\
		&-C_2\int\limits_{1}^{2} th(t)\sum\limits_{j=0}^{\A-1} a_{j} D_{s}^{j}\Bigg{\{}\frac{\Big{\{} \left[\lambda^2(1+t)^2-s^2\right]\left[\lambda^2(1-t)^2 - s^2\right]\Big{\}}^{\A-\frac{1}{2}}}{{\lambda}^{4\A}t^{2\A}}\Bigg{\}}_{|_{s=0}} \D t,
	\end{align*}
	where $C_1$ and $C_2$ are some non-zero dimensional constants whose exact values are irrelevant to us.\\
	Note that the last two integrals cancel each other out. This can be seen by taking $D_s|_{s=0}$ derivatives up to order $\A-1$ of the left and right-hand sides in the equation \eqref{RC}. One more thing to note is that the presence of $\lambda$ in the last two integrals does not make any difference as when we take the $D_s^j|_{s=0}$ derivative, $\lambda^{-2(j+1)}$ comes out of both the integrals. This can be observed from \eqref{D_s derivative}.
	
    Next, changing the order of integration in the first two integrals, we get,
	\begin{align*}
		I&= -C_1\int\limits_{0}^{1} s j_{\A}(\lambda s) \int\limits_{0}^{1-s} \frac{th(t)}{t^{2\A}} \Big{\{} \left[(1+t)^2-s^2\right]\left[(1-t)^2 - s^2\right]\Big{\}}^{\A-\frac{1}{2}} \D t \D s\\
		&+C_1\int\limits_{0}^{1} s j_{\A}(\lambda s) \int\limits_{1+s}^{2} \frac{th(t)}{t^{2\A}} \Big{\{} \left[(1+t)^2-s^2\right]\left[(1-t)^2 - s^2\right]\Big{\}}^{\A-\frac{1}{2}} \D t \D s.
	\end{align*}
	
	Since the inner integrals are same for each $0<s<1$ from \eqref{RC}, we have $I=0$. This completes the proof of the theorem.
	\end{proof}
    \subsection{Proof of sufficiency in Theorem \ref{range}}
    The sufficiency part of Theorem \ref{range} follows immediately as a consequence of the cross product-type identity stated in the previous subsection.
\begin{proof}[Proof of sufficiency in Theorem \ref{range}]
Assuming \eqref{RC}, Theorem \ref{CP-Identity} shows that  
	\[
	y_{\A}(\lambda)\int\limits_0^{2}th(t) j_{\A}(\lambda t) \D t =j_{\A}(\lambda)\int\limits_0^{2}th(t) y_{\A}(\lambda t) \D t. 
	\]
	
	This gives that \eqref{RC} is sufficient by the result from \cite{Agranovsky-Finch-Kuchment-range}; see Theorem \ref{T2.4}, using the fact that Bessel functions of the first and the second kind do not have common zeros.
\end{proof} 
    \subsection{Proof of Theorem \ref{Thm1.4}}
    In this section, we prove Theorem \ref{Thm1.4}, generalizing the range characterization results proved for the radial case to general functions. The proof is very similar to what we did in our earlier work for the odd-dimensional case. We repeat the arguments here for the sake of completeness. 
    
From the calculations done in \cite{Salman_Article}, we have the following: 
\begin{align}
	\notag g_{m,l}(t)& = \frac{\o_{n-1}}{4^{\frac{n-3}{2}} t^{n-2} \o_{n} C_{m}^{\frac{n-2}{2}}(1)}\int\limits_{|1-t|}^{1} u f_{m,l}(u)C_{m}^{\frac{n-2}{2}}\lb \frac{1+u^2-t^2}{2u}\rb\Big{\{}\lb (1+t)^2-u^2\rb\lb u^2-(1-t)^2\rb\Big{\}}^{\frac{n-3}{2}} \D u\\
	\label{Eq49.2} &=\frac{\o_{n-1}}{ t^{n-2} \o_{n}C_{m}^{\frac{n-2}{2}}(1)}\int\limits_{|1-t|}^{1} u^{n-2} f_{m,l}(u)C_{m}^{\frac{n-2}{2}}\lb \frac{1+u^2-t^2}{2u}\rb\Bigg{\{}1- \frac{\lb 1+u^2-t^2\rb^2}{4u^2}\Bigg{\}}^{\frac{n-3}{2}} \D u.
\end{align}

We use the following formula for Gegenbauer polynomials: 
\[
C_m^{\A}(x)=K (1-x^2)^{-\A+\frac{1}{2}}\frac{\D^{m}}{\D x^{m}}\lb 1-x^{2}\rb^{m+\A-\frac{1}{2}},
\]
where 
\[
K=\frac{(-1)^{m} \Gamma(\A + \frac{1}{2})\Gamma(m+2\A)}{2^m m!\Gamma(2\A)\Gamma(m+\A+\frac{1}{2})}.
\]
We have 
\Beq\label{Eq49.3}
C_{m}^{\frac{n-2}{2}}\lb \frac{1+u^2-t^2}{2u}\rb=K\lb 1-\lb \frac{1+u^2-t^2}{2u}\rb^2\rb^{-\frac{n-3}{2}}(-u)^{m} D^{m} \lb 1-\frac{\lb 1+u^2-t^2\rb^2}{4u^2}\rb^{m+\frac{n-3}{2}},
\Eeq
where $D=\frac{1}{t}\frac{\D}{ \D t }$. This follows by a repeated application of chain rule.

Substituting \eqref{Eq49.3} into \eqref{Eq49.2}, we get, 
\begin{align*}
	t^{n-2} g_{m,l}(t)=\frac{K(-1)^{m}\o_{n-1}}{\o_n C_{m}^{\frac{n-2}{2}}(1)}\int\limits_{|1-t|}^{1} u^{m+n-2}f_{m,l}(u)D^{m}\lb 1-\frac{\lb 1+u^2-t^2\rb^2}{4u^2}\rb^{m+\frac{n-3}{2}} \D u.
\end{align*}
Noting that $\A=\frac{n-2}{2}$ and that $D^{m}$ can be taken outside the integral, we get, 
\begin{align*}
	t^{n-2} g_{m,l}(t)&=\frac{K(-1)^{m}\o_{n-1}}{4^{m+\A-\frac{1}{2}}\o_n C_{m}^{\A}(1)}D^{m}\int\limits_{|1-t|}^{1} u^{1-m}f_{m,l}(u)\lb 4u^2-\lb 1+u^2-t^2\rb^2\rb^{m+\A -\frac{1}{2}} \D u\\
	&=\frac{K(-1)^{m}\o_{n-1}}{4^{m+\A-\frac{1}{2}}\o_n C_{m}^{\A}(1)}D^{m}\int\limits_{|1-t|}^{1} u^{1-m}f_{m,l}(u)\Big{\{}\lb (1+t)^2-u^2\rb\lb u^2-(1-t)^2\rb\Big{\}}^{m+\A-\frac{1}{2}} \D u.
\end{align*}
We denote 
\begin{align*}
	& h_{m,l}(t) = t^{n-2}g_{m,l}(t)\\
	&\phi_{m,l}(t)=\int\limits_{|1-t|}^{1} u^{1-m}f_{m,l}(u)\Big{\{}\lb (1+t)^2-u^2\rb\lb u^2-(1-t)^2\rb\Big{\}}^{m+\A-\frac{1}{2}} \D u.
\end{align*}
Then we have 
\[
h_{m,l}(t)=\frac{K(-1)^{m}\o_{n-1}}{4^{m+\A-\frac{1}{2}}\o_n C_{m}^{\A}(1)}D^{m}\phi_{m,l}(t).
\]
We make the following observations: 
\begin{itemize}
	\item $\phi_{m,l}(t)\in C_c^{\infty}((0,2))$,
	\item $\phi_{m,l}(t)$ satisfies the following: For $0\leq s\leq 1$,
	\Beq \label{general_range}
	\int\limits_{0}^{1-s} \frac{t\phi_{m,l}(t)}{t^{2(m+\A)}} \Big{\{}(1+t)^2-s^2)((1-t)^2 - s^2)\Big{\}}^{m+\A- \frac{1}{2}} \D t = \int\limits_{1+s}^{2} \frac{t\phi_{m,l}(t)}{t^{2(m+\A)}} \Big{\{}(1+t)^2-s^2)((1-t)^2 - s^2)\Big{\}}^{m+\A- \frac{1}{2}} \D t.
	\Eeq    
\end{itemize}
The smoothness in the first point follows from the fact that $g_{m,l}(t)$ is a smooth function and $\phi_{m,l}(t)$ is the solution of a linear ODE (the ODE being $D^m$) with smooth coefficients and with zero initial conditions. The fact that the support is strictly in $(0,2)$ is due to the fact that $f_{m,l}\in C^{\infty}([0,1))$ has support strictly away from $1$.  The second point follows from the proof of the necessity given for the radial case; see Theorem \ref{range} by replacing $\A$ by $m+\A$. Hence we have the following necessary condition: There is a function $\phi_{m,l}\in C_c^{\infty}((0,2))$, such that $h_{m,l}(t)= D^{m}\phi_{m,l}(t)$ and $\phi_{m,l}(t)$ satisfies \eqref{general_range}.
We note that for each $0\leq l\leq d_m$, $\phi_{m,l}$ satisfies the same integral equation.

Next we show that this condition is also sufficient. Since $\phi_{m,l}(t)\in C_c^{\infty}((0,2))$ and $\phi_{m,l}(t)$ satisfies \eqref{general_range}, we have by the sufficiency part of the range characterization for radial functions that 
\begin{align*}
	\lb \int\limits_0^{\infty} j_{\A+m}(\lambda t) t\phi_{m,l}(t)\D t\rb y_{\A+m}(\lambda)   =\lb \int\limits_0^{\infty} y_{\A+m}(\lambda t) t\phi_{m,l}(t) \D t\rb j_{\A+m}(\lambda).
\end{align*}
Therefore, using the fact that $j$ and $y$ both satisfy \eqref{der}, we have 
\begin{align*}
	\lb \int\limits_0^{\infty} D^{m}j_{\A}(\lambda t) t\phi_{m,l}(t)\D t\rb y_{\A+m}(\lambda)   =\lb \int\limits_0^{\infty} D^{m}y_{\A}(\lambda t) t\phi_{m,l}(t) \D t\rb j_{\A+m}(\lambda).
\end{align*}
Integrating by parts, we get, 
\begin{align*}
	\lb \int\limits_0^{\infty} j_{\A}(\lambda t) t h_{m,l}(t)\D t\rb y_{\A+m}(\lambda)   =\lb \int\limits_0^{\infty} y_{\A}(\lambda t) t h_{m,l}(t) \D t\rb j_{\A+m}(\lambda).
\end{align*}
We have the same expression for each $0\leq l\leq d_m$ and hence the $m^{\mathrm{th}}$ order spherical harmonic term of the Hankel transform of $g$ defined as the orthogonal projection of the Hankel transform of $g$ onto the subspace of spherical harmonics of degree $m$ vanishes at the non-zero zeros of the spherical Bessel function $j_{\A+m}(\lambda)$, satisfying the condition in Theorem~\ref{T2.4}. We are done with the general case as well.

\section*{Acknowledgments}
DA was supported by SERB's Overseas Visiting Doctoral Fellowship.

VPK would like to thank the Isaac Newton Institute for Mathematical Sciences, Cambridge, UK, for support and hospitality during the workshop, \emph{Rich and Nonlinear Tomography - a multidisciplinary approach} in 2023 where this work was initiated (supported by EPSRC Grant Number EP/R014604/1). Additionally, he acknowledges the support of the Department of Atomic Energy,  Government of India, under
Project No.\ 12-R\&D-TFR-5.01-0520. Finally, he would like to thank Sivaguru Ravisankar and Manmohan Vashisth for several interesting discussions regarding elliptic integrals.  

\bibliographystyle{plain}
\bibliography{references.bib}

\begin{thebibliography}{10}

\bibitem{Abramowitz-Stegun}
Milton Abramowitz and Irene~A. Stegun.
\newblock {\em Handbook of Mathematical Functions with Formulas, Graphs, and Mathematical Tables}, volume No. 55 of {\em National Bureau of Standards Applied Mathematics Series}.
\newblock U. S. Government Printing Office, Washington, DC, 1964.

\bibitem{agranovsky1996approximation}
Mark Agranovsky, Carlos Berenstein, and Peter Kuchment.
\newblock Approximation by spherical waves in ${L}^p$-spaces.
\newblock {\em J. Geom. Anal.}, 6(3):365--383, 1996.

\bibitem{Agranovsky-Finch-Kuchment-range}
Mark Agranovsky, David Finch, and Peter Kuchment.
\newblock Range conditions for a spherical mean transform.
\newblock {\em Inverse Probl. Imaging}, 3(3):373--382, 2009.

\bibitem{Agranovsky-Kuchment-Quinto}
Mark Agranovsky, Peter Kuchment, and Eric~Todd Quinto.
\newblock Range descriptions for the spherical mean {R}adon transform.
\newblock {\em J. Funct. Anal.}, 248(2):344--386, 2007.

\bibitem{agranovsky2010range}
Mark Agranovsky and Linh~V Nguyen.
\newblock Range conditions for a spherical mean transform and global extendibility of solutions of the {D}arboux equation.
\newblock {\em Journal d'Analyse Math{\'e}matique}, 112(1):351--367, 2010.

\bibitem{agranovsky1996injectivity}
Mark Agranovsky and Eric~Todd Quinto.
\newblock Injectivity sets for the {R}adon transform over circles and complete systems of radial functions.
\newblock {\em J. Funct. Anal.}, 139(2):383--414, 1996.

\bibitem{agranovsky1999conical}
Mark Agranovsky, Valery Volchkov, and Lawrence Zalcman.
\newblock Conical uniqueness sets for the spherical {R}adon transform.
\newblock {\em Bulletin of the London Mathematical Society}, 31(2):231--236, 1999.

\bibitem{AAKN1}
Divyansh Agrawal, Gaik Ambartsoumian, Venkateswaran~P. Krishnan, and Nisha Singhal.
\newblock A simple range characterization for spherical mean transform in odd dimensions and its applications.
\newblock {\em Preprint, arXiv:2310.20702}, 2023.
\newblock Submitted.

\bibitem{AAKN2}
Divyansh Agrawal, Gaik Ambartsoumian, Venkateswaran~P Krishnan, and Nisha Singhal.
\newblock On the null space of the backprojection operator and {R}ubin’s conjecture for the spherical mean transform.
\newblock {\em Inverse Problems}, 40(12):125018, 2024.

\bibitem{Ambartsoumian2018}
Gaik Ambartsoumian, Rim Gouia-Zarrad, Venkateswaran~P. Krishnan, and Souvik Roy.
\newblock Image reconstruction from radially incomplete spherical {R}adon data.
\newblock {\em European J. Appl. Math.}, 29(3):470--493, 2018.

\bibitem{Ambartsoumian-Zarrad-Lewis}
Gaik Ambartsoumian, Rim Gouia-Zarrad, and Matthew~A. Lewis.
\newblock Inversion of the circular {R}adon transform on an annulus.
\newblock {\em Inverse Problems}, 26(10):105015, 11, 2010.

\bibitem{Ambartsoumian2015}
Gaik Ambartsoumian and Venkateswaran~P. Krishnan.
\newblock Inversion of a class of circular and elliptical {R}adon transforms.
\newblock In {\em Complex analysis and dynamical systems {VI}. {P}art 1}, volume 653 of {\em Contemp. Math.}, pages 1--12. Amer. Math. Soc., Providence, RI, 2015.

\bibitem{ref:AmbKuch}
Gaik Ambartsoumian and Peter Kuchment.
\newblock {On the injectivity of the circular Radon transform}.
\newblock {\em Inverse Problems}, 21:473--485, 2005.

\bibitem{ref:AmbKuch-range}
Gaik Ambartsoumian and Peter Kuchment.
\newblock {A range description for the planar circular Radon transform}.
\newblock {\em SIAM J. Math. Anal.}, 38(2):681--692, 2006.

\bibitem{ambartsoumian2014exterior}
Gaik Ambartsoumian and Leonid Kunyansky.
\newblock Exterior/interior problem for the circular means transform with applications to intravascular imaging.
\newblock {\em Inverse Problems and Imaging}, 8(2):339--359, 2014.

\bibitem{And}
Lars-Erik Andersson.
\newblock {On the determination of a function from spherical averages}.
\newblock {\em SIAM J. Math. Anal.}, 19:214--232, 1988.

\bibitem{aramyan2020recovering}
Rafik~H. Aramyan and Robert~M. Mnatsakanov.
\newblock To recovering the moments from the spherical mean {R}adon transform.
\newblock {\em Journal of Mathematical Analysis and Applications}, 490(2):124334, 2020.

\bibitem{BF}
Paul~F. Byrd and Morris~D. Friedman.
\newblock {\em Handbook of elliptic integrals for engineers and scientists}.
\newblock Die Grundlehren der mathematischen Wissenschaften, Band 67. Springer-Verlag, New York-Heidelberg, 1971.
\newblock Second edition, revised.

\bibitem{CH_Book}
Richard Courant and David Hilbert.
\newblock {\em Methods of Mathematical Physics. {V}ol. {II} Partial {D}ifferential {E}quations}.
\newblock John Wiley \& Sons, Inc., New York, 1989.
\newblock Reprint of the 1962 original, A Wiley-Interscience Publication.

\bibitem{denisjuk1999integral}
Alexander Denisjuk.
\newblock Integral geometry on the family of semi-spheres.
\newblock {\em Fract. Calc. Appl. Anal}, 2(1):31--46, 1999.

\bibitem{Fawcett}
John Fawcett.
\newblock {Inversion of N-dimensional spherical averages}.
\newblock {\em SIAM J. Appl. Math.}, 45:336--341, 1985.

\bibitem{Finch-Haltmeir-Rakesh_even-inversion}
David Finch, Markus Haltmeier, and Rakesh.
\newblock Inversion of spherical means and the wave equation in even dimensions.
\newblock {\em SIAM J. Appl. Math.}, 68(2):392--412, 2007.

\bibitem{Finch-P-R}
David Finch, Sarah Patch, and Rakesh.
\newblock Determining a function from its mean values over a family of spheres.
\newblock {\em SIAM J. Math. Anal.}, 35:1213--1240, 2004.

\bibitem{finch2006range}
David Finch and Rakesh.
\newblock The range of the spherical mean value operator for functions supported in a ball.
\newblock {\em Inverse Problems}, 22(3):923, 2006.

\bibitem{Hrycak-Schmutzhard}
Tomasz Hrycak and Sebastian Schmutzhard.
\newblock A {N}icholson-type integral for the cross-product of the {B}essel functions.
\newblock {\em J. Math. Anal. Appl.}, 436(1):168--178, 2016.

\bibitem{John-book}
Fritz John.
\newblock {\em Plane Waves and Spherical Means Applied to Partial Differential Equations}.
\newblock Dover Publications, Inc., Mineola, NY, 2004.
\newblock Reprint of the 1955 original.

\bibitem{Krantz-Parks_primer}
Steven~G. Krantz and Harold~R. Parks.
\newblock {\em A Primer of Real Analytic Functions}.
\newblock Birkh\"{a}user Advanced Texts: Basler Lehrb\"{u}cher. [Birkh\"{a}user Advanced Texts: Basel Textbooks]. Birkh\"{a}user Boston, Inc., Boston, MA, second edition, 2002.

\bibitem{Kuchmment_2025}
P~Kuchment and L~Kunyansky.
\newblock Half-time range description for the free space wave operator and the spherical means transform.
\newblock {\em Inverse Problems}, 41(3):035005, 2025.

\bibitem{K}
Leonid~A. Kunyansky.
\newblock Explicit inversion formulae for the spherical mean {R}adon transform.
\newblock {\em Inverse Problems}, 23(1):373--383, 2007.

\bibitem{Claus_Muller}
Claus M\"{u}ller.
\newblock {\em Spherical harmonics}, volume~17 of {\em Lecture Notes in Mathematics}.
\newblock Springer-Verlag, Berlin-New York, 1966.

\bibitem{nguyen2009family}
Linh~V. Nguyen.
\newblock A family of inversion formulas in thermoacoustic tomography.
\newblock {\em Inverse Problems and Imaging}, 3(4):649--675, 2009.

\bibitem{Norton-circle}
Stephen~J. Norton.
\newblock Reconstruction of a two-dimensional reflecting medium over a circular domain: exact solution.
\newblock {\em J. Acoust. Soc. Am.}, 67:1266--1273, 1980.

\bibitem{norton1981ultrasonic}
Stephen~J. Norton and Melvin Linzer.
\newblock Ultrasonic reflectivity imaging in three dimensions: exact inverse scattering solutions for plane, cylindrical, and spherical apertures.
\newblock {\em IEEE Transactions on Biomedical Engineering}, 28(2):202--220, 1981.

\bibitem{roy2015efficient}
Souvik Roy, Venkateswaran~P Krishnan, Praveen Chandrashekar, and AS~Vasudeva Murthy.
\newblock An efficient numerical algorithm for the inversion of an integral transform arising in ultrasound imaging.
\newblock {\em Journal of Mathematical Imaging and Vision}, 53:78--91, 2015.

\bibitem{R}
Boris Rubin.
\newblock Inversion formulae for the spherical mean in odd dimensions and the {E}uler-{P}oisson-{D}arboux equation.
\newblock {\em Inverse Problems}, 24(2):025021, 10, 2008.

\bibitem{Salman_Article}
Yehonatan Salman.
\newblock Recovering functions from the spherical mean transform with limited radii data by expansion into spherical harmonics.
\newblock {\em J. Math. Anal. Appl.}, 465(1):331--347, 2018.

\bibitem{xu2004reconstructions}
Yuan Xu, Lihong~V Wang, Gaik Ambartsoumian, and Peter Kuchment.
\newblock Reconstructions in limited-view thermoacoustic tomography.
\newblock {\em Medical physics}, 31(4):724--733, 2004.

\end{thebibliography}
\end{document}